\documentclass[3p,times,final]{elsarticle}

\usepackage[utf8]{inputenc}
\usepackage[T1]{fontenc}
\usepackage{mathtools}
\usepackage{amsthm}
\usepackage{amsmath,amssymb}
\usepackage{xcolor}
\usepackage{pgf,tikz,pgfplots}
\usetikzlibrary{arrows}

\usepackage{enumerate}

\newtheorem{teo}{Theorem}[section] 
\newtheorem*{teo*}{Theorem}
\newtheorem{lem}[teo]{Lemma} 
\newtheorem{prop}[teo]{Proposition}

\newtheorem{definicion}{Definition}[section]
\newtheorem{remark}{Remark}[section]

\newcommand{\mc}{\mathcal}

\newcommand{\R}{\mathbb{R}}\newcommand{\Rn}{\R^n}\newcommand{\Rnn}{\R^{n\times n}}
\newcommand{\N}{\mathbb{N}}

\renewcommand{\a}{\alpha}

\renewcommand{\O}{\Omega}

\newcommand{\e}{\varepsilon}
\newcommand{\g}{\gamma}

\renewcommand{\d}{\delta}

\renewcommand{\t}{\theta}

\newcommand{\p}{\partial}
\newcommand{\f}{\varphi}

\newcommand{\weakc}{\rightharpoonup}
\newcommand{\dd}{\mathrm{d}}

\DeclareMathOperator{\pv}{pv}

\renewcommand{\div}{\operatorname{div}}
\DeclareMathOperator{\cof}{cof}

\DeclareMathOperator{\loc}{loc}

\DeclareMathOperator{\Div}{Div}

\DeclareMathOperator{\supp}{supp}

\DeclareMathOperator{\diver}{div}
\DeclareMathOperator{\Diver}{Div}
\usepackage{amsmath}
\usepackage{latexsym}
\usepackage{amssymb}

\begin{document}

\begin{frontmatter}
\title{Fractional Piola identity and polyconvexity in fractional spaces}

\author[ucmladdress]{Jos\'e C. Bellido\corref{mycorrespondingauthor}}
\cortext[mycorrespondingauthor]{Corresponding author}
\ead{JoseCarlos.Bellido@uclm.es}

\author[ucmladdress]{Javier Cueto}
\ead{Javier.Cueto@uclm.es}

\author[uamaddress]{Carlos Mora-Corral}
\ead{Carlos.Mora@uam.es}

\address[uamaddress]{Department of Mathematics, Universidad Aut\'onoma de Madrid, Cantoblanco, 28.049-Madrid}

\address[ucmladdress]{E.T.S.I.\ Industriales, Department of Mathematics,
Universidad de Castilla-La Mancha, 13.071-Ciudad Real, Spain}

\begin{abstract}
In this paper we address nonlocal vector variational principles obtained by substitution of the classical gradient by the Riesz fractional gradient.
We show the existence of minimizers in Bessel fractional spaces under the main assumption of polyconvexity of the energy density, and, as a consequence, the existence of solutions to the associated Euler--Lagrange system of nonlinear fractional PDE. 
The main ingredient is the fractional Piola identity, which establishes that the fractional divergence of the cofactor matrix of the fractional gradient vanishes.
This identity implies the weak convergence of the determinant of the fractional gradient, and, in turn, the existence of minimizers of the nonlocal energy. Contrary to local problems in nonlinear elasticity, this existence result is compatible with solutions presenting discontinuities at points and along hypersurfaces. 
\end{abstract}


\end{frontmatter}



\section[Introduction]{Introduction}

In the last years there has been a renewed interest in variational problems involving the so-called {\it Riesz $s$-fractional gradient}, which for a function $u\in C_c^\infty(\Rn)$ is defined as 
\begin{equation}\label{eq:Ds}
 D^s u(x)=c_{n,s} \pv_x \int_{\Rn} \frac{u(x)-u(y)}{|x-y|^{n+s}}\frac{x-y}{|x-y|}\,dy, \qquad x \in \Rn ,
\end{equation} 
where $\pv_x$ stands for the principal value centered at $x$, and $c_{n,s}$ is a suitable constant. In the recent references \cite{ShS2015,ShS2018}, variational principles for functionals depending on this fractional gradient are addressed, as well as the fractional PDE derived from those as equilibrium equations. The authors consider typical calculus of variations problems, with standard growth conditions in which the classical (local) gradient is substituted by $D^su$. This naturally leads to the consideration of the space 
\[H^{s,p}(\Rn)=\left\{ u\in L^p(\Rn) \; :\; D^s u\in L^p (\R^n, \R^{n})\right\} . \]
It is also of interest the affine subspace of functions verifying a \emph{complement value condition}; to be precise, given $g\in H^{s,p} (\Rn)$ and a bounded domain $\O\subset \Rn$, we consider
\[H_g^{s,p}(\O)=\left\{u\in H^{s,p}(\Rn)\;:\; u=g\text{ in }\O^c\right\}, \]
where $\O^c$ stands for the complement of $\O$ in $\Rn$.
References \cite{ShS2015,ShS2018} deal with integral functionals of the form 
\begin{equation} \label{nlfunc}
I(u)=\int_{\Rn} W(x,u(x) ,D^su(x))\,dx
\end{equation}
with $p$-growth conditions, and prove the existence of minimizers in $H_g^{s,p}$ under the fundamental hypothesis of convexity of $W$ in the last variable, as well as natural coercivity conditions. Taking advantage of the fractional framework and the properties of Riesz potentials and Fourier transform, they also show some remarkable results on the functional spaces $H^{s,p}$, including a fractional Sobolev-type inequality, the compact embedding into $L^p$ and the equivalence with Bessel spaces (see \cite{Adams,Stein70,RunSi}). The book \cite[Ch.\ 15]{ponce_book} also pays attention to the $s$-fractional gradient, providing a proof not based on Fourier transform of a fractional fundamental theorem of calculus, also proved in \cite{ShS2015}, which is used for showing a Sobolev type inequality from $H^{s,p}$ to $L^p$. Another reference also dealing with the fractional gradient in the case $p=1$ is \cite{Schikorra2017}. 

Even more recent references are \cite{COMI2019,Silhavy2019}. In \cite{Silhavy2019}, the $s$-fractional gradient together with the $s$-fractional divergence are studied in a systematic manner. Several important properties are given, such as the uniqueness up to a multiplicative constant of the fractional gradient under natural requirements (invariance under translations and rotations, homogeneity under dilations and some continuity properties in an appropriate functional space), and some {\it fractional calculus} rules. The results in \cite{Silhavy2019} established, both from a mathematical and physical perspective, what was pointed out earlier in \cite{ShS2015,ShS2018,ponce_book}, namely, that the $s$-fractional gradient is the natural definition for a fractional differential object. We agree with the previous references on the claim that the $s$-fractional gradient deserves more attention in the literature, and likely there will be both more theoretical studies and applications in different contexts. In \cite{COMI2019}, it is addressed the space of functions $u$ whose total fractional variation is finite, naturally leading to the definition of the space $BV^s(\Rn)$ and to a $s$-fractional Caccioppoli perimeter concept. Several interesting results are shown, including a continuous embedding of fractional Sobolev spaces into $BV^s$, a Sobolev-type inequality, a coarea formula, a $s$-fractional isoperimetric inequality and a natural $s$-fractional analogue of De Giorgi's notion of reduced boundary.

The definition of the fractional gradient \eqref{eq:Ds} has an important drawback when thinking about certain applications, for instance, in nonlocal solid mechanics, since it requires an integration in the whole space for the computation of the gradient.
This is somehow meaningless in those contexts and, in addition, makes it difficult the extension to more realistic boundary conditions. In \cite{MeS}, a definition of a nonlocal gradient in bounded or unbounded domains, for which \eqref{eq:Ds} is a particular case, is introduced. Moreover, the localization of this nonlocal gradient when the nonlocality parameter ({\it horizon}) $\delta$ goes to zero is analyzed, showing the convergence to the local gradient in different topologies. For a domain $\O\subset \R^n$, the nonlocal gradient is defined as
\begin{equation}\label{eq:nonlocal_gradient}
D_{\O,\d} u(x)= \pv_x \int_\O \frac{u(x)-u(y)}{|x-y|}\frac{x-y}{|x-y|}\rho_\delta(x-y)\, dy,
\end{equation}
where the function $\rho_\delta$, typically radial, is an integral kernel reflecting the forces of interaction exerted by particles separated by a positive distance smaller than $\delta$. Typically $\rho_\delta$ vanishes outside the ball of radius $\delta$ centered at the origin. The oldest references we are aware of regarding nonlocal gradients of this kind used for models in continuum mechanics are \cite{EdelenLaws1971,EdelenGreenLaws1971}, where a nonlocal model of elasticity was proposed. More recently, a nonlocal alternative theory in solid mechanics, named peridynamics, has been proposed in \cite{SaS,SILLING201073}. The development of peridynamics in the last years has been impressive; however, most of the work until now concerns linear elastic models. In \cite{MeD, MeD16}, a linear elasticity model in the context of peridynamics is rigorously studied, proving the existence and uniqueness of solutions and their convergence to the local Lam\'e system of linear elasticity as the horizon goes to zero. In these references a projected version of nonlocal gradient \eqref{eq:nonlocal_gradient} is used in the context of the so-called state based peridynamic model. Another interesting approach going from a nonlocal peridynamic framework with finite horizon to a nonlocal fractional situation with infinite horizon in the linear case, appears in the recent references \cite{Men-fractional-1,Men-fractional-2,Men-fractional-3}. It is also worth mentioning \cite{EvBe}, where both the ill- and well-posedness of the classical Eringen model of nonlocal elasticity are addressed.  

In this investigation we deepen in the existence issue for vector variational problems involving the $s$-fractional gradient, as well as the PDE derived from those as equilibrium conditions. Thus, we consider the more difficult vectorial case under conditions weaker than convexity. To be precise, we establish the existence of minimizers in $H^{s,p}$ under the polyconvexity assumption of the integrand. Polyconvexity is a central notion in calculus of variations, with essential implications on the existence and stability of solutions in solid mechanics, and particularly in elasticity \cite{Ball77, dacorogna}. In order to obtain our results, we follow the usual steps as for classical polyconvex variational problems, namely, we show that the determinant (or any minor) of the fractional gradient matrix $D^s u$ is continuous with respect to weak convergence in $H^{s,p}$. Similarly to the classical case, we need a fractional version of the Piola identity, highlighting this as the key ingredient and the most remarkable contribution of this work. In this new situation we follow \cite{MeS} to define a fractional divergence and establish an integration by parts formula (see also \cite{ponce_book, Silhavy2019}).
We adapt the techniques developed there for some spaces of nonlocal type to our $H^{s,p}$ spaces.
More concretely, as in \eqref{eq:Ds}, the Riesz $s$-fractional gradient of a vector field $u : \Rn \to \Rn$ is 
\[ 
 D^s u(x)=c_{n,s} \pv_x \int_{\Rn} \frac{u(x)-u(y)} {|x-y|^{n+s}} \otimes \frac{x-y}{|x-y|}\,dy, \qquad x \in \Rn ,\]
where $\otimes$ denotes the tensor product, and the Riesz $s$-fractional divergence of a vector field $\mathbf{\psi} : \Rn \to \Rn$ is defined as
\[\diver^s \mathbf{\psi} (x) =-c_{n,s} \pv_x \int_{\Rn} \frac{\mathbf{\psi}(x)+\mathbf{\psi}(y)}{|x-y|^{n+s}}\cdot \frac{x-y}{|x-y|}\,dy , \qquad x \in \Rn . \]
As mentioned above, we prove the fractional Piola identity $\Div^s \cof D^s u = 0$ (where $\Div^s$ means the $s$-divergence by rows), then we show the weak continuity of $\det D^su$, the weak lower semicontinuity of polyconvex functionals in $H^{s,p}$, and finally we settle the existence of minimizers for \eqref{nlfunc}. We believe that the fractional Piola identity is a result of interest in itself. On the one hand, it may serve to show analogous versions in the fractional or nonlocal situations of classical results in whose proof the Piola identity is invoked, as for instance, the change of variables formula for surface integrals. On the other hand, it may also be useful in other fractional or nonlocal models in different contexts, such as fluid mechanics \cite{DuTian}. Furthermore, an extension to a nonlocal Piola identity for nonlocal gradients defined on bounded domains (see \eqref{eq:nonlocal_gradient}) is easy from the proof we provide here in the fractional framework. 

In the last decade there has been a great deal of work on fractional PDE of elliptic type involving the fractional Laplacian in some way.
Our results here enlarge this theory by giving an existence result of minimizers of nonlinear fractional vector variational problems based on polyconvexity, which implies, in turn, an existence result of solutions to nonlinear fractional PDE systems. The amount of references on nonlocal equations and fractional Laplacian is overwhelming, so for situations related to this work we just cite the survey \cite{Ros-Oton}, the paper \cite{ShS2018} and the references therein. On the other hand, we would like to point out the relationship of our study with nonlocal elasticity and peridynamics.
As mentioned above, the variational principle considered in this paper is not an appropriate model in solid mechanics, but a version in bounded domains of the functional \eqref{nlfunc} involving the nonlocal gradient \eqref{eq:nonlocal_gradient}, satisfying additional requirements in order to be physically consistent, fits into the peridynamics state-based model for large deformations \cite{SILLING201073}. Whereas in $H^{s,p}(\Rn)$, the structural functional analysis facts necessary to prove an existence theory for functionals like \eqref{nlfunc} are known (continuous and compact embeddings into $L^p$), those are still unknown for an analogous version of this space in bounded domains. In this sense, and since the proof provided for the fractional Piola identity may be adapted in a more or less straightforward way to bounded domains, we think this study may be seen as a first step towards a rigorous mathematical theory of nonlocal hyperelasticity. Furthermore, one primary interest for us is that $H^{s,p}$ is larger than $W^{1,p}$, and functions in $H^{s,p}$ may exhibit singularities prohibited in $W^{1,p}$, as we point out in Section \ref{se:functional}. We would like to emphasize that, contrary to classical elasticity, both singularities along points (cavitation) and hypersurfaces are compatible with the existence of solutions in $H^{s,p}$ (see Theorem \ref{th:existence}). This seems to indicate that the $L^p$ norm of $D^s u$ not only contributes to the elastic energy, but also to a kind of surface energy, since the latter is necessary in the modelling of such singularities (see, e.g., \cite{Ball01,MuSp,DaFrTo05,HeMo10}).

The outline of the paper is the following.
Section \ref{se:functional} introduces the functional space of fractional type $H^{s,p}(\Rn)$ and its main properties. We also include examples of functions exhibiting singularities belonging to these spaces but not to Sobolev spaces.
Section \ref{se:preliminaries} contains some calculus facts in $H^{s,p}$, such as the formulas for the fractional derivative of a product and the fractional integration by parts.
In Section \ref{se:Piola} we prove the fractional Piola identity. 
Section \ref{se:weak} shows the weak continuity of minors in $H^{s,p}$.
Finally, in Section \ref{se:existence} we prove the existence of minimizers of \eqref{nlfunc} for polyconvex integrands, and obtain the associated Euler--Lagrange system of fractional PDE.

\section[Functional analysis framework]{Functional analysis framework}\label{se:functional}

This section introduces general properties of the functional space $H^{s,p}$.
We start by setting the definition of principal value.
Given a function $f : \Rn \to \R$ and $x \in \Rn$ such that $f \in L^1 (B(x,r)^c)$ for every $r>0$, we define the principal value centered at $x$ of $\int_{\Rn} f$, denoted by
\[
 \pv_{x} \int_{\mathbb{R}^n}f \quad \text{or} \quad \pv_{x}\int f ,
\]
as
\begin{equation*}
 \lim_{r \rightarrow 0} \int_{B(x,r)^c} f,
\end{equation*}
whenever this limit exists.
We have denoted by $B(x,r)$ the open ball centered at $x$ of radius $r$, and by $B(x,r)^c$ its complement.
As most integrals in this work are over $\Rn$, we will use the symbol $\int$ as a substitute for $\int_{\Rn}$.

The $s$-fractional gradient is defined as follows.

\begin{definicion}\label{de:Ds}
Let $u : \Rn \to \R$ be a measurable function.
Let $0<s<1$ and $x \in \Rn$ be such that
\begin{equation}\label{eq:uufinite}
 \int_{B(x, r)^c} \frac{|u(x)-u(y)|}{|x-y|^{n+s}} dy < \infty 
\end{equation}
for each $r>0$.
Set
\begin{equation*}
 c_{n,s}= - (n+s-1) \frac{\Gamma \left(\frac{n+s-1}{2} \right)}{\pi^{\frac{n}{2}} \, 2^{1-s} \, \Gamma \left(\frac{1-s}{2}\right)} ,
\end{equation*}
where $\Gamma$ is the Euler gamma function.
We define $D^s u (x)$, the $s$-fractional gradient of $u$ at $x$, as
\begin{equation*}
D^s u(x):=c_{n,s} \pv_{x}\int \frac{u(x)-u(y)}{|x-y|^{n+s}}\frac{x-y}{|x-y|}dy ,
\end{equation*}
whenever the principal value exists.
\end{definicion}
We note that, due to symmetry, 
\[
 \pv_{x} \int \frac{x-y}{|x-y|^{n+s+1}}dy=0,
\]
and consequently, the equality
\begin{equation}\label{Alternative gradient def}
   D^su(x) = -c_{n,s} \, \pv_{x}\int \frac{u(y)}{|x-y|^{n+s}}\frac{x-y}{|x-y|}dy
\end{equation}
holds.
Notice also that the constant $c_{n,s}$ is negative.

Definition \ref{de:Ds} naturally extends to vector fields. Given $u : \Rn \to \R^m$ measurable such that \eqref{eq:uufinite} holds for each $r>0$, its $s$-fractional gradient is 
\begin{displaymath}
D^su(x)=c_{n,s} \pv_{x} \int \frac{u(x)-u(y)}{|x-y|^{n+s}}\otimes\frac{x-y}{|x-y|}dy,
\end{displaymath}
whenever it exists.
Here, $\otimes$ stands for the tensor product of vectors. 
Given $s\in(0,1)$ and $p\in (1,\infty)$, we define the space $H^{s,p}$ as
\begin{displaymath}
 H^{s,p}(\mathbb{R}^n, \mathbb{R}^m):=\{ u \in  L^p(\mathbb{R}^n): D^su \in L^p(\mathbb{R}^n, \mathbb{R}^{n \times m}) \},
\end{displaymath}
and we denote $H^{s,p}(\Rn)=H^{s,p}(\Rn,\R)$.
We will study \emph{complement value problems} (as in, for example, \cite{FKV}), so we are interested in the case in which functions are prescribed in the complement of a given set.
Thus, given an open set $\O\subset \Rn$ and $g\in H^{s,p}(\Rn,\mathbb{R}^m)$, the space $H^{s,p}_g(\O,\mathbb{R}^m)$ is defined as 
\begin{equation}\label{eq:Hg}
H_{g}^{s,p}(\Omega, \mathbb{R}^m):=\{ u \in H^{s,p}(\mathbb{R}^n, \mathbb{R}^m): u= g \text{ in } \Omega^c \}.
\end{equation}

The space $H^{s,p}$, together with the $s$-fractional gradient as a mathematical object, was introduced and studied in \cite{ShS2015,ShS2018} (see also \cite[Sect.\ 15.2]{ponce_book}).
The first remarkable fact is the identification of $H^{s,p}$ with the classical Bessel potential spaces (see \cite{Adams, Stein70, RunSi}) established in \cite[Th.\ 1.7]{ShS2015}.
Thanks to this equivalence, and rewriting well-known properties for Bessel spaces in terms of $H^{s,p}$ spaces, we obtain several basic properties that we summarize in the following proposition (see \cite[Ch.\ 7, p.\ 221]{Adams}).
We denote by $\hookrightarrow$ continuous inclusion.

\begin{prop} \label{Theorem properties H^{s,p}}
Set $0<s<1$ and $1<p<\infty$. Then:
\begin{enumerate}[a)]
\item $C_{c}^{\infty}(\mathbb{R}^n)$ is dense in $H^{s,p}(\mathbb{R}^n)$. 

\item
$H^{s,p}(\mathbb{R}^n)$ is reflexive.

\item\label{item:Hspembed} If $s< t< 1$ and $1<q\leq p\leq \frac{nq}{n-(t-s)q}$, then $H^{t,q}(\mathbb{R}^n) \hookrightarrow H^{s,p}(\mathbb{R}^n)$.

\item\label{item:HspMorrey} If $0 < \mu \leq s-\frac{n}{p}$, then $H^{s,p}(\mathbb{R}^n)\hookrightarrow C^{0,\mu}(\mathbb{R}^n)$.

\item If $p=2$, then $H^{s,2}(\mathbb{R}^n)=W^{s,2}(\mathbb{R}^n)$ with equivalence of norms.

\item\label{item:Hspfractional} If $0<s_1<s<s_2 < 1$ then $H^{s_2,p}(\mathbb{R}^n) \hookrightarrow W^{s,p}(\mathbb{R}^n) \hookrightarrow H^{s_1,p}(\mathbb{R}^n)$.
\end{enumerate}
\end{prop}

We have denoted by $W^{s,p}$ the classical fractional Sobolev space.
Although they will not be used in this paper, they were mentioned in Proposition \ref{Theorem properties H^{s,p}} to help locate the spaces $H^{s,p}$ in a more familiar scale of regularity.
We have also denoted by $C^{0,\mu}$ the space of H\"older continuous functions of exponent $\mu$.

An essential tool for obtaining existence of minimizers for variational functionals is a Poincar\'e-type inequality. Collecting together several theorems present in the literature, we state the following result, which is not optimal but suitable for our purposes.
Henceforth, given $p > 1$ and $0<s<1$ with $sp<n$ we define $p^* := \frac{np}{n-sp}$.

\begin{teo}\label{th:PoincareSobolev}
Set $0<s<1$ and $1<p<\infty$. Let $\O\subset \Rn$ be a bounded open set. Then there exists $C=C(|\O|,n,p,s)>0$ such that
\begin{displaymath}
\lVert u \rVert_{L^q (\O)} \leq C \lVert D^s u \rVert_{L^p(\mathbb{R}^n)}
\end{displaymath}
for all $u \in H^{s,p}(\mathbb{R}^n)$, and any $q$ satisfying
\[
\begin{cases}
 q\in \left[1, p^* \right] & \text{if } sp<n , \\
 q\in[1, \infty) & \text{if } sp=n , \\
 q\in[1,\infty] & \text{if } sp>n .
\end{cases}
\]
\end{teo}

The case $sp<n$ is an immediate consequence of \cite[Th.\ 1.8]{ShS2015}, where the continuous embedding of $H^{s,p}(\mathbb{R}^n)$ in $L^{p^*}(\Rn)$ is shown.
Case $sp=n$ is a consequence of \cite[Th.\ 1.10]{ShS2015}, where it is proved in this context the version of Trudinger's inequality, which implies the embedding of $H^{s,p}(\mathbb{R}^n)$ in $L^q_{\loc} (\Rn)$ for all $q \in [1, \infty)$. 
Finally, the case $sp>n$ is a consequence of Proposition \ref{Theorem properties H^{s,p}}\,\emph{\ref{item:HspMorrey})}.

The following result decides which of the embeddings of Theorem \ref{th:PoincareSobolev} are compact.
We will indicate by $\weakc$ weak convergence.

\begin{teo} \label{Bessel embedding theorem}
Set $0<s<1$ and $1<p<\infty$. Let $\Omega\subset \mathbb{R}^n$ be open and bounded and $g \in H^{s,p}(\mathbb{R}^n)$. Then for any sequence $\{u_j \}_{j \in \N} \subset  H_{g}^{s,p}(\Omega)$
such that
\begin{equation*}
u_j \rightharpoonup u \quad \text{in } H^{s,p} (\Rn),
\end{equation*}
for some $u \in H^{s,p} (\Rn)$, one has $u \in H^{s,p}_g (\O)$ and
\begin{equation*}
u_j \rightarrow u \quad \text{in } L^q (\Rn),
\end{equation*}
for every $q$ satisfying
\[
\begin{cases}
 q\in \left[1, p^* \right) & \text{if } sp<n , \\
 q\in[1, \infty) & \text{if } sp = n , \\
 q\in[1, \infty] & \text{if } sp > n .
\end{cases}
\] 
\end{teo}

Case $sp<n$ is actually \cite[Th.\ 2.2]{ShS2018}. Case $sp=n$ follows from the former having in mind Proposition 2.1, part \emph{\ref{item:Hspembed})} or else part \emph{\ref{item:Hspfractional})}. Finally, the case $sp>n$ is a consequence of Proposition \ref{Theorem properties H^{s,p}}\,\emph{\ref{item:HspMorrey})} and the compact embedding of $C^{0,\mu} (\bar{\O})$ into $C (\bar{\O})$.

It is worth mentioning that case $p=1$ is intentionally avoided. First, because the original statements of Proposition \ref{Theorem properties H^{s,p}} and Theorems \ref{th:PoincareSobolev} and \ref{Bessel embedding theorem} exclude this case, and, second, because we are concerned with a general existence theory that requires reflexive spaces.

\subsection[Examples of functions in $H^{s,p}(\mathbb{R}^n)$]{Examples of functions in $H^{s,p}(\mathbb{R}^n)$}

One of the motivations of this study is to propose an existence theory formulated in spaces wider than classical Sobolev spaces.
As a consequence of Proposition \ref{Theorem properties H^{s,p}}\,\emph{\ref{item:Hspfractional})}, classical Sobolev spaces are continuously embedded in $H^{s,p}$ spaces.
Further, we are interested in functions that belong to $H^{s,p}$ but not to $W^{1,p}$.
Necessarily, those functions must exhibit some type of singularity.
We focus on two important singularities in solid mechanics: discontinuities along hypersurfaces and at a single point. The later corresponds with the paradigmatic case of cavitation.
For simplicity, we study as a model for singularities along hypersurfaces a function whose first component is the characteristic function $\chi_Q$ of the unit cube $Q$, while the other components are $C^{\infty}_c$ functions.
As a model for singularity at a point, we study a radial function of compact support exhibiting one cavity at the origin.
In both examples the functions have compact support: this simplifies the analysis since it avoids the issue of the integrability at infinity, and, hence, allows us to focus solely on the singularity.

We start with the case of singularity along a hypersurface.
There is an extensive literature on when the characteristic function of a set (especially, of an open bounded Lipschitz set) belongs to a functional space of fractional regularity (see, e.g., \cite{Triebel83,RunSi,sickel,Mazja11,FaracoRogers}).
We exploit those results to give a quick proof of the following lemma.

\begin{lem}\label{le:fracture}
Set $0<s<1$ and $1<p<\infty$.
Let $Q=(0,1)^n$ and $\f_2, \ldots, \f_n \in C_c^\infty(\mathbb{R}^n)$.
Define $u=(\chi_Q, \f_2, \ldots, \f_n)$.
Then
\[
 u \in H^{s,p}(\mathbb{R}^n,\mathbb{R}^n) \ \text{ if } \ p < \frac{1}{s}, \qquad \text{and} \qquad u \notin H^{s,p}(\mathbb{R}^n,\mathbb{R}^n) \ \text{ if } \ p>\frac{1}{s} .
\]
\end{lem}
\begin{proof}
As $C_c^\infty(\mathbb{R}^n) \subset H^{s,p}(\mathbb{R}^n)$ (we will show this in Lemma \ref{Lema difference quotient bound}), we have that $u \in H^{s,p}(\mathbb{R}^n,\mathbb{R}^n)$ if and only if $\chi_Q \in H^{s,p}(\mathbb{R}^n)$.

The fractional Sobolev space $W^{s,p}$ coincides with the Triebel--Lizorkin space $F^s_{p,p}$ and with the Besov space $B^s_{p,p}$ (see, e.g., \cite[Sect.\ 2.3.5]{Triebel83} or \cite[Prop.\ 2.1.2]{RunSi}).
This result together with \cite[Lemma 4.6.3.2]{RunSi} shows that $\chi_Q \in W^{s,p}$ if and only if $s p < 1$.
Proposition \ref{Theorem properties H^{s,p}}\,\emph{\ref{item:Hspfractional})} concludes the proof.
\end{proof}

For the case of cavitation, the result is the following.

\begin{lem}\label{le:cavitation}
Set $0<s<1$ and $1<p<\infty$.
Let $\f \in C^{\infty}_c ([0,\infty))$ be such that $\f(0) > 0$, and $u(x)=\frac{x}{|x|} \f(|x|)$.
Then
 \[ u \in H^{s,p}(\mathbb{R}^n,\mathbb{R}^n)  \ \text{ if } \ p < \frac{n}{s} , \qquad \text{and} \qquad u \notin H^{s,p}(\mathbb{R}^n,\mathbb{R}^n)  \ \text{ if } \ p> \frac{n}{s} .
\]
\end{lem}
 \begin{proof}
It is well known that $u \in W^{1,q}(\mathbb{R}^n,\mathbb{R}^n)$ whenever $1<q<n$ (see, e.g., \cite[Lemma 4.1]{Ball82}), and therefore $u\in H^{t,q}(\mathbb{R}^n,\mathbb{R}^n)$ for any $0<t<1$ and $1<q<n$.
Applying now Proposition \ref{Theorem properties H^{s,p}}\,\emph{\ref{item:Hspembed})}, we have that $u\in H^{s, p} (\mathbb{R}^n,\mathbb{R}^n)$ for any $s\in (0,t)$ and $p \in [q, \frac{nq}{n - (t-s)q}]$.
Now we observe that the set of $(s,p) \in \R^2$ such that there exist $q \in (1,n)$ and $t \in (0,1)$ for which $s \in (0, t)$ and $p \in [q, \frac{nq}{n - (t-s)q}]$ is precisely the set of $(s,p)$ such that $s \in (0,1)$ and $p \in (1, \frac{n}{s})$.
Therefore, $u \in H^{s,p}(\mathbb{R}^n,\mathbb{R}^n)$  if $p < \frac{n}{s}$.
 
On the other hand, when $p > \frac{n}{s}$, by Proposition \ref{Theorem properties H^{s,p}}\,\emph{\ref{item:HspMorrey})}, $H^{s,p}(\mathbb{R}^n,\mathbb{R}^n)$ functions are continuous.
Since $u$ is discontinuous, $u\notin H^{s,p}(\mathbb{R}^n,\mathbb{R}^n)$ if $p>\frac{n}{s}$.
 \end{proof}
 %


\section{Calculus in $H^{s,p}$}\label{se:preliminaries}

In this section we present some calculus rules of nonlocal functionals related to $H^{s,p}$, notably, an integration by parts formula.

We start with a sufficient condition for the $s$-fractional gradient to be defined everywhere.

\begin{lem} \label{Lema difference quotient bound}
Let $0< \alpha < s <1$ and $\varphi \in C^{0,\alpha}(\mathbb{R}^n) \cap C^{0,1}(\mathbb{R}^n)$.
Then
\begin{equation} \label{eq:infty norm bound}
 \sup_{x \in \Rn} \int \frac{\left| \varphi(x)-\varphi(y) \right|}{|x-y|^{n+s}} dy <\infty. 
\end{equation}
If, in addition, $\varphi$ has compact support then $D^s \varphi \in L^r(\Rn)$, for every $r \in [1,\infty]$.
\end{lem}

\begin{proof}
Let $L$ and $C$ be, respectively, the Lipschitz and $\alpha$-H\"older constants of $\varphi$.
Then, for every $x \in \mathbb{R}^n$,
\begin{align*}
\int \frac{\left| \varphi(x)-\varphi(y) \right|}{|x-y|^{n+s}} \, dy &\leq \int_{B(x,1)} \frac{L}{|x-y|^{n+s-1}} \, dy+ \int_{B(x,1)^c} \frac{C}{|x-y|^{n+s-\alpha}} \, dy 
= \int_{B(0,1)} \frac{L}{|z|^{n+s-1}} \, dz + \int_{B(0,1)^c} \frac{C}{|z|^{n+s-\alpha}} \, dz< \infty .
\end{align*}

This means that $D^s\varphi \in L^\infty(\Rn)$.

Next we are going to see that $D^s\varphi \in L^1(\Rn)$ when $\varphi$ has compact support.
Denote by $F$ the support of $\varphi$.
Then
\begin{equation*}
 \int\left| c_{n,s}\int \frac{\varphi(x)-\varphi(y)}{|x-y|^{n+s}}\frac{x-y}{|x-y|} dy\right| dx \leq |c_{n,s}| \left( A + B \right) ,
\end{equation*}
where
\[
A :=\int \int_{F} \frac{\left| \varphi(x)-\varphi(y) \right|}{|x-y|^{n+s}} dy \, dx, \qquad
B :=\int \int_{F^c} \frac{\left| \varphi(x)-\varphi(y) \right|}{|x-y|^{n+s}} dy \, dx .
\]
Now, we observe that, applying Fubini's Theorem and \eqref{eq:infty norm bound},
\begin{equation} \label{eq:integrability of smooth functions 1}
A = \int_{F}\int  \frac{\left| \varphi(x)-\varphi(y) \right|}{|x-y|^{n+s}}  dx \, dy < \infty .
\end{equation}
We notice that $\left| \varphi(x)-\varphi(y) \right|=0$ for every $(x,y) \in F^c \times F^c$. Therefore, applying again \eqref{eq:infty norm bound} we get
\begin{equation}\label{eq:integrability of smooth functions 2}
B =\int_{F}\int_{F^c} \frac{\left| \varphi(x)-\varphi(y) \right|}{|x-y|^{n+s}}  dy \, dx \leq \int_{F}\int \frac{\left| \varphi(x)-\varphi(y) \right|}{|x-y|^{n+s}}  dy \, dx < \infty.
\end{equation}
As a consequence of \eqref{eq:integrability of smooth functions 1} and \eqref{eq:integrability of smooth functions 2}, $D^s \varphi \in L^1(\Rn)$.
Finally, through a standard interpolation argument, we get that $D^s\varphi \in L^r(\Rn)$ for all $r \in [1,\infty]$.
\end{proof}
Lemma \ref{Lema difference quotient bound} implies, in particular, that $D^s \f$ is defined everywhere for $\varphi \in C^{0,\alpha}(\mathbb{R}^n) \cap C^{0,1}(\mathbb{R}^n)$ and $0< \alpha < s <1$.
It also shows that $C^1_c (\Rn) \subset H^{s,p} (\Rn)$ for every $0 < s < 1$ and $1 < p < \infty$.

The following result defines a nonlocal operator related to the $s$-fractional gradient.

\begin{lem} \label{Lema operador lineal} 
Let $1\leq q < \infty$ and $0 < \alpha < s< 1$.
Let $\varphi \in C^{0,\alpha}(\mathbb{R}^n) \cap C^{0,1}(\mathbb{R}^n)$ and $k \in \N$.
Then, the operator $K_{\varphi}: L^q(\mathbb{R}^n, \mathbb{R}^{k \times n}) \rightarrow L^q(\mathbb{R}^n, \mathbb{R}^k)$ defined as
\begin{equation*}
 K_{\varphi}(U)(x)= c_{n,s} \int \frac{\varphi(x)-\varphi(y)}{|x-y|^{n+s}} U(y) \frac{x-y}{|x-y|} dy , \qquad \text{a.e. } x \in \mathbb{R}^n ,
\end{equation*}
is linear and bounded.

If, in addition, $\varphi$ has compact support then $K_{\varphi}$ is bounded from $L^q(\mathbb{R}^n, \mathbb{R}^{k \times n})$ to $ L^p(\mathbb{R}^n, \mathbb{R}^k)$ for all $p \in [1,q]$.
\end{lem}
\begin{proof}
Let $U \in L^q(\mathbb{R}^n, \mathbb{R}^{k \times n})$.
We denote by $C$ a positive constant that does not depend on $U$ and whose value may vary along the proof.
For a.e.\ $x \in \mathbb{R}^n$ we have
\[
 \left| K_{\varphi} (U) (x) \right| \leq |c_{n,s}| \int  \frac{\left| \varphi(x)-\varphi(y) \right|}{|x-y|^{n+s}} \left| U(y) \right| dy ,
\]
so
\begin{equation}\label{eq:K1}
 \left| K_{\varphi} (U) (x) \right|^q \leq 2^{q-1}|c_{n,s}|^q \left( g(x) + h(x) \right) ,
\end{equation}
with
\[
 g(x) := \left(\int_{B(x,1)} \frac{\left| \varphi(x)-\varphi(y) \right|}{|x-y|^{n+s}} \left| U(y) \right| dy \right)^q, \qquad
 h(x) :=\left(\int_{B(x,1)^c} \frac{\left| \varphi(x)-\varphi(y) \right|}{|x-y|^{n+s}} \left| U(y) \right| dy \right)^q.
\]
Let $L$ be a Lipschitz and an $\alpha$-H\"older constant of $\varphi$.
Then, using that $\varphi$ is Lipschitz and applying H\"older's inequality, we get
\begin{align*}
 g(x) & \leq L^q \left(\int_{B(x,1)} \frac{|U(y)|}{|x-y|^{n+s-1}}dy \right)^q = L^q \left(\int_{B(0,1)}  \frac{|U(x-z)|}{|z|^{n+s-1}}dz \right)^q \\
 & \leq L^q \int_{B(0,1)} \frac{|U(x-z)|^q}{|z|^{n+s-1}}dz \left(\int_{B(0,1)} \frac{1}{|z|^{n+s-1}}dz\right)^{q-1} = C \int_{B(0,1)}  \frac{|U(x-z)|^q}{|z|^{n+s-1}}dz .
\end{align*}
Integrating,
\begin{equation}\label{eq:K2}
\begin{split}
 \int  g(x)\,dx & \leq C\int  \int_{B(0,1)}  \frac{|U(x-z)|^q}{|z|^{n+s-1}}dz \,dx =C \int_{B(0,1)}\frac{1}{|z|^{n+s-1}} \int  |U(x-z)|^q dx \,dz  \\
 &= C\left\| U\right\|_{L^q(\mathbb{R}^n, \R^{k \times n})}^q \int_{B(0,1)}\frac{1}{|z|^{n+s-1}}  dz \leq C \left\| U\right\|_{L^q(\mathbb{R}^n, \R^{k \times n})}^q . 
\end{split}
\end{equation}

As for the term $h$, using that $\varphi$ is $\alpha$-H\"older and applying H\"older's inequality,
\begin{align*}
 h(x) & \leq L^q \left( \int_{B(x,1)^c} \frac{|U(y)|}{|x-y|^{n+s-\alpha}} dy \right)^q  \\
& \leq L^q \int_{B(0,1)^c}\frac{\left| U(x-z) \right|^q}{|z|^{n+s-\alpha}} dz \left(\int_{B(0,1)^c}\frac{1}{|z|^{n+s-\alpha}}dz\right)^{{q-1}} \leq  C\int_{B(0,1)^c}\frac{\left| U(x-z) \right|^q}{|z|^{n+s-\alpha}} dz.
\end{align*}
Integrating,
\begin{equation}\label{eq:K3}
\begin{split}
\int h(x)\,dx&=C\int \int_{B(0,1)^c} \frac{| U(x-z)|^q}{|z|^{n+s-\alpha}}dz  \,dx = C \int_{B(0,1)^c}\frac{1}{|z|^{n+s-\alpha}} \int | U(x-z)|^qdx  \,dz  \\
&= C \left\| U\right\|_{L^q(\mathbb{R}^n, \R^{k \times n})}^q\int_{B(0,1)^c}\frac{1}{|z|^{n+s-\alpha}} dz \leq C \left\| U\right\|_{L^q(\mathbb{R}^n, \R^{k \times n})}^q .
\end{split}
\end{equation}
Putting together \eqref{eq:K1}, \eqref{eq:K2} and \eqref{eq:K3} we obtain
 \begin{equation*}
 \left\| K_{\varphi}(U)\right\|_{L^q(\mathbb{R}^n, \R^k)}^{q} \leq C  \left\| U\right\|_{L^q(\mathbb{R}^n, \R^{k \times n})}^q,
 \end{equation*}
and the first part of the proof is finished.

Next we are going to see that the linear operator $K_\varphi:L^q(\mathbb{R}^n, \R^{k \times n}) \rightarrow L^1(\mathbb{R}^n, \R^k)$ is bounded in the case $\varphi$ has compact support.
Denote by $F$ the support of $\varphi$.
Then
\begin{equation}\label{eq:2K1}
 \int \left| K_{\varphi} (U) (x) \right| dx \leq |c_{n,s}| \left( A + B \right) ,
\end{equation}
where
\[
A :=\int \int_{F} \frac{\left| \varphi(x)-\varphi(y) \right|}{|x-y|^{n+s}} \left| U(y) \right| dy \, dx, \qquad
B :=\int \int_{F^c} \frac{\left| \varphi(x)-\varphi(y) \right|}{|x-y|^{n+s}} \left| U(y) \right| dy \, dx .
\]
Now, we observe that, applying Fubini's Theorem, H\"older's inequality and Lemma \ref{Lema difference quotient bound},
\begin{equation}\label{eq:2K2}
A \leq \int_{F}\left| U(y) \right|\int  \frac{\left| \varphi(x)-\varphi(y) \right|}{|x-y|^{n+s}}  dx \, dy\leq C\left(\int_{F}\left| U(y)\right|^q dy\right)^{1/q} \leq C \left\| U\right\|_{L^q(\mathbb{R}^n, \R^{k \times n})} .
\end{equation}
We notice that $\left| \varphi(x)-\varphi(y) \right|=0$ for every $(x,y) \in F^c \times F^c$. Therefore, applying H\"older's inequality and Lemma \ref{Lema difference quotient bound} we get
\begin{align*}
B &=\int_{F}\int_{F^c} \frac{\left| \varphi(x)-\varphi(y) \right|}{|x-y|^{n+s}} \left| U(y) \right| dy \, dx \\
&\leq \int_{F}\left(\int_{F^c} \frac{\left| \varphi(x)-\varphi(y) \right|}{|x-y|^{n+s}}dy\right)^{1/q'}\left(\int_{F^c}\frac{\left| \varphi(x)-\varphi(y) \right|}{|x-y|^{n+s}}\left| U(y) \right|^q dy\right)^{1/q}dx \\
&\leq C  \int_{F}\left(\int_{F^c}\frac{\left| \varphi(x)-\varphi(y) \right|}{|x-y|^{n+s}}\left| U(y) \right|^q dy\right)^{1/q}dx.
\end{align*}
Using again H\"older's inequality, Lemma \ref{Lema difference quotient bound} and Fubini's Theorem, we obtain
\begin{equation}\label{eq:2K3}
\begin{split}
B &\leq C  \left(\int_{F}\int_{F^c}\frac{\left| \varphi(x)-\varphi(y) \right|}{|x-y|^{n+s}}\left| U(y) \right|^q dy \, dx\right)^{1/q} \left|F \right|^{1/{q'}} \\
&=C\left(\int_{F^c}\left| U(y) \right|^q \int_{F}\frac{\left| \varphi(x)-\varphi(y) \right|}{|x-y|^{n+s}} dx \, dy \right)^{1/q}\leq
C\left(\int_{F^c}\left| U(y) \right|^qdy \right)^{1/q} \leq C \left\| U\right\|_{L^q(\mathbb{R}^n, \R^{k \times n})} ,
\end{split}
\end{equation}
where $|F|$ denotes the measure of $F$.
Inequalities \eqref{eq:2K1}, \eqref{eq:2K2} and \eqref{eq:2K3} lead us to
\[
\left\| K_{\varphi}(U) \right\|_{L^1(\mathbb{R}^n, \R^k)} \leq C \left\| U \right\|_{L^q(\mathbb{R}^n, \R^{k \times n})}.
\]
Finally, through a standard interpolation argument, we get that $K_{\varphi }$ is bounded from $L^q(\mathbb{R}^n, \R^{k \times n})$ to $ L^p(\mathbb{R}^n, \R^k)$ for all $p \in [1,q]$.
\end{proof}

As a consequence of Lemma \ref{Lema operador lineal} and a general result, the operator $K_\varphi$ is continuous from the weak topology of $L^q(\mathbb{R}^n, \mathbb{R}^{k \times n})$ to the weak topology of $L^q(\mathbb{R}^n, \mathbb{R}^k)$ and, in the case of a $\f$ of compact support, from the weak topology of $L^q(\mathbb{R}^n, \mathbb{R}^{k \times n})$ to the weak topology of $L^p(\mathbb{R}^n, \mathbb{R}^k)$ for all $p \in [1,q]$.

The next lemma finds out the spaces where the sequence $\{D^s u_j\}$ is convergent, provided that $\{ u_j\}$ is convergent in $H^{s,p}$.
\begin{lem} \label{le:convergence of bounded supported functions} 
Let $0 < s < 1$ and $1 < p < \infty$. 
Let $u \in H^{s,p}(\Rn)$ and let $\{ u_j \}_{j\in \N} \subset H^{s,p}(\Rn)$ be a sequence converging to $u$ in $H^{s,p}(\Rn)$.
Assume that there is a compact $K \subset \Rn$ such that $\bigcup_{j=1}^{\infty} \supp u_j \subset K$.
Then $D^s u_j \to D^s u$ in $L^r(\Rn)$ for every $r \in [1,p]$.
\end{lem}
\begin{proof}
By linearity, we can assume that $u = 0$.
Call $K_B=K+B(0,1)$.
Then
\begin{equation}\label{eq:bounded supports convergence 1}
\left\| D^s u_j \right\|_{L^1(\Rn)} \leq \left\| D^s u_j \right\|_{L^{p}(K_B)} |K_B|^{\frac{1}{p'}} + \left\| D^s u_j \right\|_{L^1(K_B^c)} ,
\end{equation}
where $|K_B|$ denotes the Lebesgue measure of $K_B$, and $p'$ is the conjugate exponent of $p$.

On the other hand, for every $j \in \N$, we use Fubini's Theorem and H\"older's inequality to get
\begin{equation}\label{eq:L1KcB}
\begin{split}
\left\| D^su_j  \right\|_{L^1(K_B^c)}& = |c_{n,s}|\int_{K_B^c} \left| \int_{K} \frac{u_j (x)-u_j (y)}{|x-y|^{n+s}}\frac{x-y}{|x-y|}dy\right|dx \leq |c_{n,s}|
\int_{K} \int_{K_B^c} \frac{|u_j (x)-u_j (y)|}{|x-y|^{n+s}} \, dx \, dy \\
&\leq |c_{n,s}|
\int_{K} \left(\int_{K_B^c} \frac{|u_j (x)-u_j (y)|^p}{|x-y|^{n+s}} \, dx\right)^{\frac{1}{p}} \left(\int_{K_B^c}\frac{1}{|x-y|^{n+s}} \,dx\right)^{\frac{1}{p'}} dy .
\end{split}
\end{equation}
Now, for every $y \in K$ we have $K_B^c - y \subset B(0,1)^c$, so
\[
 \int_{K_B^c} \frac{1}{|x-y|^{n+s}} \,dx = \int_{K_B^c - y} \frac{1}{|z|^{n+s}} \,dz \leq \int_{B(0,1)^c} \frac{1}{|z|^{n+s}} \,dz <\infty .
\]
Now, we will use $C$ to denote a constant (depending on $n$, $s$ and $K$) which can vary through the proof.
So, continuing from \eqref{eq:L1KcB} and applying H\"older's inequality again, we obtain
\[
\left\| D^su_j  \right\|_{L^1(K_B^c)} \leq C \left(\int_{K} \int_{K_B^c} \frac{|u_j (x)-u_j (y)|^p}{|x-y|^{n+s}} \, dx dy\right)^{\frac{1}{p}}  
\leq C\left\| u_j  \right\|_{W^{\frac{s}{p},p}(\Rn)}\leq C\left\| u_j  \right\|_{H^{s,p}(\Rn)},
\]
where we have used Proposition \ref{Theorem properties H^{s,p}}\,\emph{\ref{item:Hspfractional})} in the last step.
This inequality, together with \eqref{eq:bounded supports convergence 1}, leads to
\begin{equation*}
\left\| D^s u_j \right\|_{L^1(\Rn)} \leq C \left\| u_j \right\|_{H^{s,p}(\Rn)} \to 0 ,
\end{equation*}
by assumption.
Finally, through a standard interpolation argument, we obtain the convergence $D^s u_j \to 0$ in $L^r(\Rn)$ for every $r \in [1,p]$.
\end{proof}

Now we introduce a product formula for the $s$-fractional gradient.
We denote by $I$ the identity matrix of dimension $n$.

\begin{lem} \label{Gradiente producto}
Let $0 < s < 1$ and $1 < p < \infty$.
 Let $g \in H^{s,p}(\mathbb{R}^n)$ and $\f \in C_c^1(\mathbb{R}^n)$. 
 Then $\f g \in H^{s,p} (\Rn)$ and for a.e.\ $x \in \mathbb{R}^n$,
 \begin{displaymath}
 D^s(\f g)(x) = \f(x) D^s g(x) + K_{\f} (g I) (x).
 \end{displaymath}
 \end{lem} 
 
 \begin{proof}
Clearly $\f g \in L^p (\Rn)$.
Now, for a.e.\ $x \in \Rn$ we have
\begin{align*}
 D^s (\f g) (x) & = c_{n,s} \pv_{x} \int \frac{(\f g)(x) - (\f g)(y)}{|x-y|^{n+s}}\frac{x-y}{|x-y|}dy \\
 &= c_{n,s} \pv_{x} \int \frac{\f(x) g(x) - \f(x) g(y) + \f(x) g(y) - \f(y) g(y)}{|x-y|^{n+s}}\frac{x-y}{|x-y|}dy \\
 &= \f(x) D^s g(x)+ K_{\f} (g I) (x) .
 \end{align*}
The term $\f \, D^s g$ is in $L^p (\Rn, \Rn)$ since $\f \in C^1_c (\Rn)$, while the term $K_{\f} (g I)$ is in $L^p (\Rn, \Rn)$ by Lemma \ref{Lema operador lineal}.
\end{proof}
 
Inspired by \cite{MeS} (see also \cite{ponce_book}) we introduce the $s$-fractional divergence $\div^s$.

\begin{definicion}\label{determinante divergencia}
Let $\phi : \mathbb{R}^n \to \mathbb{R}^n$ be measurable.
Let $0<s<1$ and $x \in \Rn$ be such that
\[
 \int_{B(x, r)^c} \frac{|\phi(x)+\phi(y)|}{|x-y|^{n+s}} dy < \infty 
\]
for each $r>0$.
The $s$-fractional divergence of $\phi$ is defined as
\[
 \diver^s \phi(x):= - c_{n,s} \pv_{x} \int\frac{\phi(x)+\phi(y)}{|x-y|^{n+s}}\cdot\frac{x-y}{|x-y|}dy , 
\]
whenever the principal value exists.
\end{definicion}
Analogously, if $M : \Rn \to \Rnn$ is such that its rows satisfy the assumptions of Definition \ref{determinante divergencia}, we denote by $\Diver^s M$ the column vector-function whose components are the $s$-fractional divergences of each row of $M$.

Similarly to what happened with the $s$-fractional gradient (see \eqref{Alternative gradient def}), by symmetry, we have that
   \begin{equation}\label{eq:alternativediv}
   \diver^s\phi(x) =-c_{n,s}\pv_{x}\int\frac{\phi(y)}{|x-y|^{n+s}} \cdot\frac{x-y}{|x-y|}dy.
   \end{equation}

An initial property of the $s$-fractional divergence is the following, which states that it is well defined if and only if so is the $s$-fractional derivative.
Its proof is analogous to that of \cite[Lemma 2.3]{MeS}, and, hence, it will be omitted.

\begin{lem}\label{le:divgrad}
Let $u: \Rn \to \R$ be measurable and let $x \in \Rn$ be such that $u(x)$ is defined and finite.
Then
\[
 \pv_{x} \int\frac{u(x) + u(y)}{|x-y|^{n+s}} \frac{x_i - y_i}{|x-y|} \, dy
\]
exists and is finite if and only if 
\[
 \pv_x \int \frac{u(x)-u(y)}{|x-y|^{n+s}} \frac{x_i - y_i}{|x-y|}\,dy
\]
exists and is finite.
Moreover, in this case,  
\[
 - c_{n,s} \pv_{x} \int\frac{u(x) + u(y)}{|x-y|^{n+s}} \frac{x_i - y_i}{|x-y|} \, dy = c_{n,s} \pv_x \int \frac{u(x)-u(y)}{|x-y|^{n+s}} \frac{x_i - y_i}{|x-y|}\,dy .
\]
\end{lem}


The most important fact relating the $s$-fractional gradient and the $s$-fractional divergence is the integration by parts formula.
The proof of this result follows the lines of \cite[Th.\ 1.4]{MeS}.

\begin{teo} \label{Integracion por partes NL}
Let $0 < s < 1$.
Let $u \in L^1_{\loc} (\mathbb{R}^n)$ be such that
\begin{equation}\label{eq:uxuy}
 \int_K \int \frac{|u(y)-u(x)|}{|x-y|^{n+s}} \, dy \, dx < \infty .
\end{equation}
for every compact $K \subset \Rn$.
Then $D^s u \in L^1_{\loc} (\Rn, \Rn)$ and for all $\phi \in C_{c}^1(\mathbb{R}^n,\mathbb{R}^n)$,
\begin{displaymath}
\int D^su(x) \cdot \phi(x) \, dx = -\int u(x) \diver^s \phi(x) \, dx.
\end{displaymath}
\end{teo}

\begin{proof}
Assumption \eqref{eq:uxuy} implies that $D^s u$ exists a.e.\ as a Lebesgue integral and $D^s u \in L^1_{\loc} (\Rn, \Rn)$, we have
\begin{equation}\label{eq:parts1}
\int  D^su(x) \cdot \phi(x) \, dx =c_{n,s}  \int \int  \frac{u(x)-u(y)}{|x-y|^{n+s}} \frac{x-y}{|x-y|}  \cdot \phi(x) \, dy\,dx .
\end{equation}
On the other hand, as $\phi \in C_{c}^1(\mathbb{R}^n,\mathbb{R}^n)$, by Lemma \ref{le:divgrad},
\begin{equation}\label{eq:parts2}
-\int u(x) \diver^s\phi(x) dx
=c_{n,s} \int \int u(x) \frac{\phi(x)+\phi(y)}{|x-y|^{n+s}}\cdot\frac{x-y}{|x-y|} \, dy \, dx. 
\end{equation}
Thus, it suffices to establish the equality of the right hand sides of \eqref{eq:parts1} and \eqref{eq:parts2}; in fact, we will establish the equality of the double integrals in the domain $D_{\d} : = \{ (x, y) \in \R^n \times \Rn : |x-y| \geq \d \}$ for each $\d > 0$.
We have
\begin{equation*}
\iint\limits_{D_{\d}} \frac{u(x)-u(y)}{|x-y|^{n+s}} \frac{x-y}{|x-y|} \cdot \phi(x) \, dy\,dx
= 
\iint\limits_{D_{\d}}  \frac{u(x) \, \phi(x)}{|x-y|^{n+s}}\cdot\frac{x-y}{|x-y|} \, dy \, dx - \iint\limits_{D_{\d}}  \frac{u(y) \, \phi(x)}{|x-y|^{n+s}} \cdot\frac{x-y}{|x-y|} \, dy\,dx. 
\end{equation*}
If we interchange now the roles of $x$ and $y$ in the second integral, using the symmetry of $D_{\d}$, we have
\begin{equation*}
 - \iint\limits_{D_{\d}}  \frac{u(y) \, \phi(x)}{|x-y|^{n+s}} \cdot\frac{x-y}{|x-y|} \, dy\,dx = \iint\limits_{D_{\d}}  \frac{u(x) \, \phi(y)}{|x-y|^{n+s}} \cdot \frac{x-y}{|x-y|} \, dy\,dx,
\end{equation*}
and therefore
\[
 \iint\limits_{D_{\d}}  \frac{u(x)-u(y)}{|x-y|^{n+s}} \frac{x-y}{|x-y|} \cdot \phi(x) \, dy\,dx = 
 \iint\limits_{D_{\d}} u(x) \frac{\phi(x)+\phi(y)}{|x-y|^{n+s}} \cdot\frac{x-y}{|x-y|}  \, dy\,dx ,
\]
whence the equality of the right hand sides of \eqref{eq:parts1} and \eqref{eq:parts2} follows.
\end{proof}

As in Lemma \ref{Gradiente producto}, the following result computes the $s$-fractional divergence of a product.
\begin{lem} \label{Divergencia producto}
Let $0 < s < 1$ and $1 < p < \infty$.
Let $g \in H^{s,p}(\mathbb{R}^n,\mathbb{R}^n)$ and $\f \in C_c^1(\mathbb{R}^n)$.
Then $\f g \in H^{s,p}(\mathbb{R}^n,\mathbb{R}^n)$ and for a.e.\ $x \in \mathbb{R}^n$,
\begin{displaymath}
 \diver^s (\f g) (x) = \f(x) \diver^s g(x)+ K_{\f} (g^T) (x) .
 \end{displaymath}
\end{lem}


\section[Fractional Piola Identity]{Fractional Piola Identity}\label{se:Piola}

In this section we introduce a fractional version of the Piola Identity.
This is the main step in order to prove the existence of  solutions for our fractional energy, since it will allow us to prove the weak continuity in $H^{s,p}$ of the determinant of the $s$-fractional gradient.
Recall that the classical Piola identity asserts that, for smooth enough functions $u : \Omega \subset \Rn \to \Rn$ one has $\Div \cof Du= 0$.
Of course, $\cof$ denotes the cofactor matrix, which satisfies $\cof A \, A^T = (\det A) \, I$ for every $A \in \Rnn$. 

Contrary to the classical case, the proof of the fractional Piola identity is not trivial even for smooth functions. Indeed, the classical proof cannot be reproduced in this case as it relies on Leibniz's rule and symmetry of second derivatives. Notice that Lemma 3.7 prevents all the terms of the second derivatives from being cancelled as happens in the classical case. In the next lines we sketch a possible proof of this identity in order to find out the difficulties. {We emphasize that the next argument is formal in order to illustrate the difficulties for proving the fractional Piola identity.} The main assumption we make is that all integrals involved are absolutely convergent without the need of the principal value, so that we can apply Fubini's theorem. 
Starting from \eqref{Alternative gradient def}, we have
\[
 D^s u (x) = - c_{n,s} \int \frac{u(y)}{|x-y|^{n+s+1}} \otimes (x-y) \, \dd y
\]
For simplicity in the calculations, we set $n=2$, the simplest case.
Then the first row of $\cof D^s u (x)$ is
\[
 - c_{n,s} \left( \int \frac{u_2 (y)}{|x-y|^{n+s+1}} (x_2 - y_2) \, \dd y , \ - \int \frac{u_2(y)}{|x-y|^{n+s+1}} (x_1 - y_1) \, \dd y \right) ,
\]
and, using \eqref{eq:alternativediv}, the first component of $\Div^s \cof D^s u (x)$ is
\begin{align*}
 & c_{n,s}^2 \int \left[ \int \frac{u_2 (z)}{|y-z|^{n+s+1}} (y_2 - z_2) \, \dd z \, (x_1 - y_1) - \int \frac{u_2(z)}{|y-z|^{n+s+1}} (y_1 - z_1) \, \dd z \, (x_2 - y_2) \right] \dd y \\
 & = c_{n,s}^2 \int u_2 (z) \int \frac{\det \left( x-y , y-z \right)}{|y-z|^{n+s+1}} \, \dd y \, \dd z .
\end{align*}
Now, for every $x$ and $z$,
\begin{equation}\label{eq:fakeproof}
{\pv_x }\int \frac{\det \left( x-y , y-z \right)}{|y-z|^{n+s+1}} \, \dd y ={\pv_x}  \int \frac{\det \left( x-y-z , y \right)}{|y|^{n+s+1}} \, \dd y = {\pv_x} \int \frac{\det \left( x-z , y \right)}{|y|^{n+s+1}} \, \dd y = 0
\end{equation}
because of the odd symmetry of the integral. {Notice that the integrals in \eqref{eq:fakeproof} are not defined as Lebesgue integrals.} 
The real proof will consist in making these calculations rigorous for arbitrary dimension $n$. The underlying reason of why the fractional Piola identity is true is that $\det D^su$ is a sort of null Lagrangian in the sense that, for any $n\ge 2$, the integral
\[{\pv} \int\frac{\det(x-a_1,\dots,x-a_n)}{|x-a_1|^{n+s+1}\cdots |x-a_n|^{n+s+1}}\,dx \]
is zero. This is a consequence of the fact that the determinant is an alternating multilinear form, as well as that $\det Du$ is a classical null Lagrangian. However, as we will see in Lemma 4.1, the previous integral is not defined as a proper integral but as a principal value centered at points $a_1,\dots,a_n$, and this will cause technical difficulties in the proof.\\

We start by reviewing a version of the change of variables formula for surface integrals (see, e.g., \cite[Prop.\ 2.7]{MuSp}).
Let $\Gamma$ be an oriented $(n-1)$-dimensional manifold with continuous unit normal field $\nu$.
Let $T : \mathbb{R}^n \to  \mathbb{R}^n$ be affine and injective, with corresponding linear map $\vec T$.
Let $g: \mathbb{R}^n \rightarrow \mathbb{R}^n$ be smooth.
Then
\[
 \int_{\Gamma} g (T x) \cdot \cof \vec T \nu (x) \, dS(x) = \int_{T (\Gamma)} g (x) \cdot \frac{\cof \vec T \nu (T^{-1} x)}{| \cof \vec T \nu (T^{-1} x) |} \, dS(x) ,
\]
where $dS$ denotes the surface element.
Now assume that $T$ is a symmetry across a hyperplane, so $T^{-1} = T$, $\det \vec T = -1$ and $\vec T^{-1} = \vec T = \vec T^T = - \cof \vec T$.
Therefore,
\[
 - \int_{\Gamma} g(Tx) \cdot \vec T \nu(x) \, dS(x) = - \int_{T (\Gamma)} g(x)\cdot \vec T \nu (T x) \, dS(x).
\]
Thus
\[
 \int_{\Gamma} \vec T g(Tx)\cdot \nu(x) \, dS(x)=\int_{T (\Gamma)} \vec T g(x) \cdot \nu(T x) \, dS(x).
\]
As this is true for every $g$, we have that
\begin{equation} \label{eq:changevar sym}
\int_{\Gamma} g(x)\cdot \nu(x) \, dS(x) = \int_{T (\Gamma)} g(T x) \cdot \nu(T x) \, dS(x) ,
\end{equation}
which is the formula we will use in Lemma \ref{Determinante Piola}.

In this and the next sections we will employ the following notation for the submatrices.

\begin{definicion}\label{de:submatrix}
Let $k \in \N$ be with $1 \leq k \leq n$.
Consider indices $1 \leq i_1 < \cdots < i_k \leq n$ and $1 \leq j_1 < \cdots < j_k \leq n$.
\begin{enumerate}[a)]
\item\label{item:M}
We define $M=M_{i_1 , \ldots , i_k; j_1 , \ldots , j_k} : \R^{n \times n} \to \R^{k \times k}$ as the map such that $M(F)$ is the submatrix of $F \in \R^{n \times n}$ formed by the rows $i_1 , \ldots , i_k$ and the columns $j_1 , \ldots , j_k$.

\item
We define $\bar{M} = \bar{M}_{i_1 , \ldots , i_k; j_1 , \ldots , j_k} : \R^{k \times k} \to \Rnn$ as the map such that $\bar{M}(F)$ is the matrix whose rows $i_1 , \ldots , i_k$ and columns $j_1 , \ldots , j_k$ coincide with those of $F$, whereas the rest of the entries are zero.

\item
We define $N = N_{i_1 , \ldots , i_k} : \Rn \to \R^k$ as the map such that $N(v)$ is the subvector of $v \in \Rn$ formed by the entries $i_1 , \ldots , i_k$.

\item
We define $\bar{N} = \bar{N}_{i_1 , \ldots , i_k} : \R^k \to \Rn$ as the map such that $\bar{N} (v)$ is the vector whose entries $i_1 , \ldots , i_k$ coincide with those of $v$, whereas the rest of the entries are zero.

\item
We define $\tilde{N} = \tilde{N}_{i_1 , \ldots , i_k} : \Rn \to \Rn$ as $\bar{N} \circ N$.
\end{enumerate}
\end{definicion}

The following formulas for the determinant will be useful.
Given $A \in \Rnn$, we express it as
\begin{equation*}
 A = \begin{pmatrix} a_1 \\ \vdots \\ a_n \end{pmatrix} ,
\end{equation*}
where $a_1, \ldots, a_n \in \mathbb{R}^n$ are its rows.
Then $\det A = a_i \cdot (\cof A)_i$ for each $i \in \{1, \ldots, n\}$, where $(\cof A)_i$ denotes the $i$-th row of $\cof A$.
Now we realize that if $b \in \mathbb{R}^n$ and
\[
 A' = \begin{pmatrix} a_1 \\ \vdots \\ a_{i-1} \\ b \\ a_{i+1} \\ \vdots \\ a_n \end{pmatrix} ,
\]
then 
\begin{equation}\label{Cof Det}
 \det A' = (\cof A)_i \cdot b .
\end{equation}

The following lemma is the rigorous version of \eqref{eq:fakeproof}.

\begin{lem} \label{Determinante Piola}
Let $k \in \N$ be with $1 \leq k \leq n$.
Consider indices $1 \leq j_1 < \cdots < j_k \leq n$ and let $N = N_{j_1 , \ldots , j_k}$ be the function of Definition \ref{de:submatrix}.
Then there exists a continuous function $G : [0, \infty) \times(\mathbb{R}^n)^{k-1} \to \mathbb{R}$ such that for any $a_1, \ldots, a_k \in \mathbb{R}^n$ and $\epsilon_1, \ldots, \epsilon_k>0$ we have
\begin{equation*}
 \left| \int_{\left(\bigcup_{j=1}^{k} B(a_j, \epsilon_j)\right)^c} \frac{ \det ( N(x-a_{1}), \ldots, N(x-a_{k}))}{|x-a_{1}|^{n+s+1}\cdots|x-a_{k}|^{n+s+1}} dx \right| \leq \frac{\epsilon_1^{1-s}}{\left(\epsilon_2 \cdots \epsilon_k \right)^{n+s+2}}G(\epsilon_1, a_2 - a_1, \ldots, a_k -a_1).
\end{equation*}
\end{lem}

\begin{proof}
We can assume that the points $a_1, \ldots, a_k$ do not lie on an affine manifold of dimension $k-2$, since otherwise $\det ( N(x-a_{1}),\ldots, N(x-a_{k})) = 0$ for all $x \in \Rn$.

Define $h: \mathbb{R}^n \setminus \{0\} \to \mathbb{R}$ as
\begin{equation}\label{eq:h}
 h(x)=\frac{-1}{(n+s-1)|x|^{n+s-1}}
\end{equation}
and $h_i : \mathbb{R}^n \setminus \{a_i\} \to \mathbb{R}$ as $h_i (x)= h(x-a_i)$, for each $i=1, \ldots, k$.
Define $H : \mathbb{R}^n \setminus \{a_1, \ldots, a_k\} \to \mathbb{R}^k$ componentwise as $H=(h_1,\ldots,h_k)^T$.
Then
\begin{equation}\label{eq:DH}
 D H (x) = \begin{pmatrix}
 \nabla h_1 (x) \\ \vdots \\ \nabla h_k (x)
 \end{pmatrix}
 = \begin{pmatrix}
  \frac{x-a_1}{|x-a_1|^{n+s+1}} \\ \vdots \\ \frac{x-a_k}{|x-a_k|^{n+s+1}} \end{pmatrix} .
\end{equation}
Call $\vec \jmath = (j_1, \ldots, j_k)$ and denote by $D_{\vec \jmath} H$ the submatrix of $DH$ formed by the columns $j_1, \ldots, j_k$.
Then, for all $x \in \Rn \setminus \{ a_1, \ldots, a_k \}$,
\begin{equation}\label{eq:detDH}
 \det D_{\vec \jmath} H (x) = \frac{\det (N(x-a_{1}), \ldots, N(x-a_{k}))}{|x-a_{1}|^{n+s+1}\cdots|x-a_{k}|^{n+s+1}} .
\end{equation}
As $DH \in L^p \bigl( \bigl( \bigcup_{j=1}^{k} B(a_j, \epsilon_j) \bigr)^c , \Rnn \bigr)$ for all $p \in [1, \infty]$, we have $\det D_{\vec \jmath} H \in L^1 \bigl( \bigl( \bigcup_{j=1}^{k} B(a_j, \epsilon_j) \bigr)^c \bigr)$.
Therefore,
\[
 \int_{\left(\bigcup_{j=1}^{k} B(a_j, \epsilon_j)\right)^c} \det D_{\vec \jmath} H = \lim_{R \to \infty} \int_{B(0, R) \setminus\bigcup_{j=1}^{k} B(a_j, \epsilon_j)} \det D_{\vec \jmath} H .
\]
As $H$ is smooth outside $\bigcup_{j=1}^{k} B(a_j, \epsilon_j)$, we have that 
\[
 \det D_{\vec \jmath} H = \diver \bar{N}(h_1 (\cof D_{\vec \jmath} H))_1) ,
\]
where $(\cof D_{\vec \jmath} H)_1$ indicates the first row of $\cof D_{\vec \jmath} H$, and $\bar{N} = \bar{N}_{j_1 , \ldots , j_k}$ is the function of Definition \ref{de:submatrix}.
Let $R>0$ be big enough so that $\bigcup_{j=1}^{k} \bar{B}(a_j, \epsilon_j) \subset B(0, R)$.
Then, by the divergence theorem,
\begin{equation*}
\int_{B(0, R) \setminus\bigcup_{j=1}^{k} B(a_j, \epsilon_j)} \det D_{\vec \jmath} H = - \int_{\partial\bigcup_{j=1}^{k} B(a_j, \epsilon_j)}  \bar{N}(h_1 (\cof D_{\vec \jmath} H)_1) \cdot \nu_j 
  + \int_{\partial B(0, R)}  \bar{N}(h_1 (\cof D_{\vec \jmath} H)_1) \cdot \nu_R ,
\end{equation*}
where $\nu_j (x)= \frac{x-a_j}{\epsilon_j}$ in $\partial B(a_j, \epsilon_j)$ for $j=1, \ldots, k$, and $\nu_R (x)= \frac{x}{R}$ in $\partial B(0, R)$.
Having in mind the expressions \eqref{eq:h} and \eqref{eq:DH}, we find that, for some constant $C>0$,
\[
 \left| \int_{\partial B(0, R)} \bar{N}(h_1 (\cof D_{\vec \jmath} H)_1) \cdot \nu_R \right| \leq \frac{C}{R^{(n+s)k -1}} ,
\]
which goes to zero as $R \rightarrow \infty$.
Therefore,
\begin{equation}\label{eq:intpartialunion}
 \int_{\left(\bigcup_{j=1}^{k} B(a_j, \epsilon_j)\right)^c} \det D_{\vec \jmath} H = - \int_{\partial\bigcup_{j=1}^{k} B(a_j, \epsilon_j)} \bar{N}(h_1 (\cof D_{\vec \jmath} H)_1) \cdot \nu_j  .
\end{equation}

For each $i=1, \ldots, n$ we set
\[
 A_i=\partial \left( \bigcup_{j=1}^{k} B(a_j, \epsilon_j) \right) \cap \partial B(a_i,\epsilon_i) .
\]
As a consequence of the inclusion $\partial \bigcup_{j=1}^{k} B(a_j, \epsilon_j) \subset \bigcup_{j=1}^{k} \partial B(a_j, \epsilon_j)$, we have that
\[
 \partial \bigcup_{j=1}^{k} B(a_j, \epsilon_j)=\bigcup_{j=1}^{k} A_j .
\]
Moreover, the $(n-1)$-dimensional area of $A_i \cap A_j$ is zero for $1 \leq i < j \leq k$.
Figure \ref{fig:esferas} illustrates this situation when $k=n=3$.

\begin{figure}
 \centering 
 \includegraphics[height=0.25\textwidth]{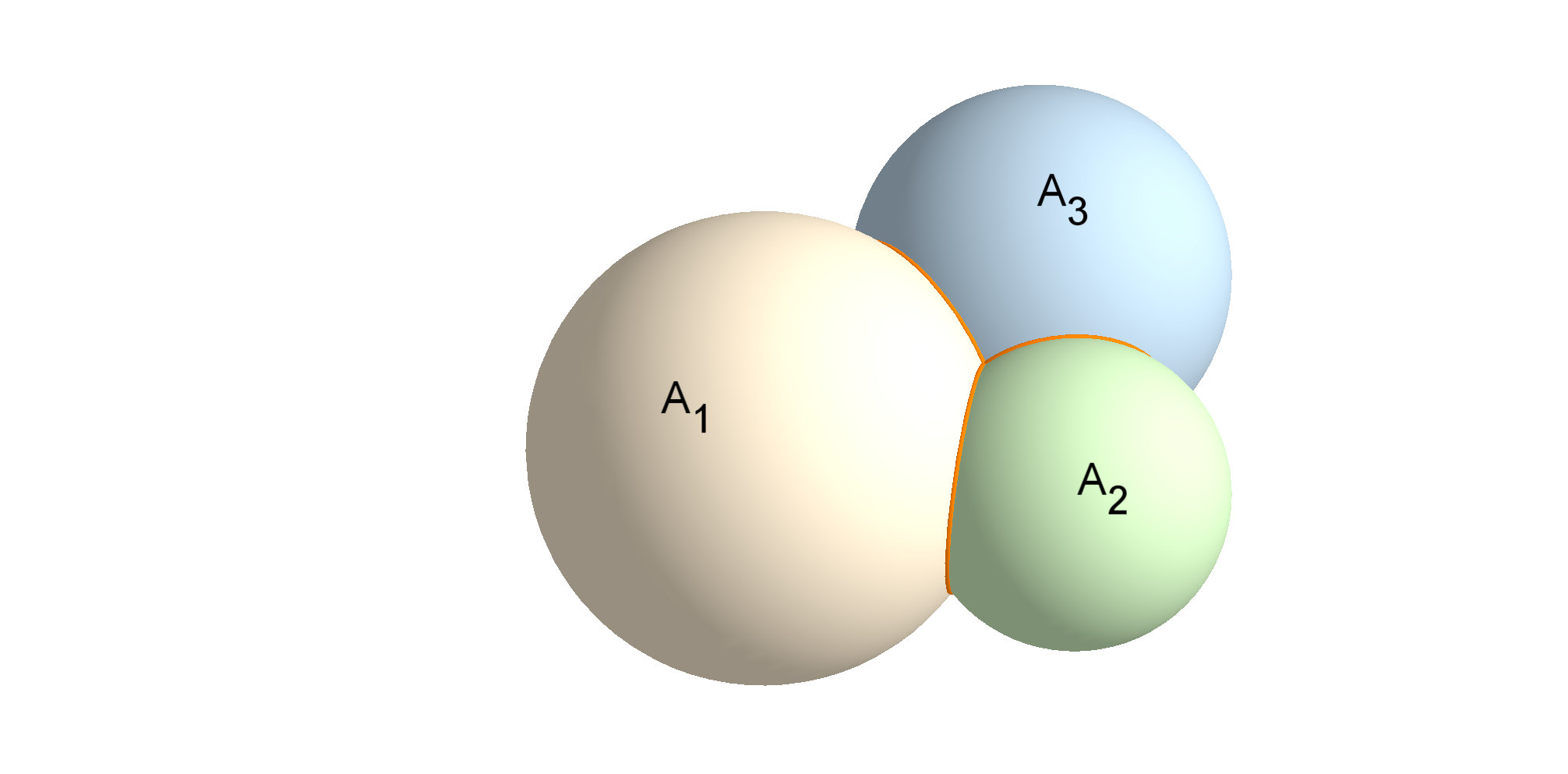}
 \caption{Sets $A_1$, $A_2$, $A_3$ in $\R^3$}
 \label{fig:esferas}
\end{figure}

Next, using \eqref{Cof Det} and \eqref{eq:DH}, we have that for $j=2, \ldots, k$ and $x \in \partial B(a_j, \epsilon_j)$,
\[
 \bar{N}(h_1 (\cof D_{\vec \jmath} H)_1) \cdot \nu_j (x)= \frac{\det ( N(x-a_j), N(x-a_2), \ldots, N(x-a_k))}{| x-a_j | \left| x-a_2 \right|^{n+s+1} \cdots \left| x-a_k \right|^{n+s+1}} = 0 .
\]
As a result, recalling \eqref{eq:intpartialunion} and the inclusion $A_j \subset \partial B(a_j, \epsilon_j)$, we have that
\begin{equation}\label{eq:detDHintA1}
 \int_{\left(\bigcup_{j=1}^{k} B(a_j, \epsilon_j) \right)^c} \det D_{\vec \jmath} H \, dx = - \int_{A_1} \bar{N}(h_1 (\cof D_{\vec \jmath} H)_1) \cdot \nu_1 \, dS .
\end{equation}

Having in mind the expression \eqref{eq:h}, the multilinearity of the determinant and considering \eqref{Cof Det} and \eqref{eq:DH}, we have that, for $x \in A_1$,
\begin{equation}\label{eq:h1cofDhnu1}
\begin{split}
- \bar{N}(h_1 (\cof D_{\vec \jmath} H)_1) \cdot \nu_1 (x) & = \frac{1}{n+s-1} \frac{1}{\epsilon_1^{n+s}} (\cof D_{\vec \jmath} H)_1 \cdot N(x - a_1) \\
&=\frac{1}{n+s-1} \frac{1}{\epsilon_1^{n+s}} \frac{\det (N(x-a_1), N(x-a_2), \ldots, N(x-a_k))}{ \left| x-a_2 \right|^{n+s+1} \cdots \left| x-a_k \right|^{n+s+1}}  \\
&=\frac{1}{n+s-1} \frac{1}{\epsilon_1^{n+s}} \frac{\det (N(x-a_1), N(a_1-a_2), \ldots, N(a_1-a_k))}{ \left| x-a_2 \right|^{n+s+1} \cdots \left| x-a_k \right|^{n+s+1}} \\
&=\frac{1}{n+s-1} \frac{1}{\epsilon_1^{n+s-1}} \frac{(\bar{M}(\cof (N(x-a_1), N(a_1-a_2), \ldots, N(a_1-a_k))))_1}{ \left| x-a_2 \right|^{n+s+1} \cdots \left| x-a_k \right|^{n+s+1}} \cdot \nu_1 (x) ,
\end{split}
\end{equation}
where $\bar{M} = \bar{M}_{i_1 , \ldots , i_k; j_1 , \ldots , j_k}$ is the function of Definition \ref{de:submatrix}.

Let $\Pi_k$ be the only hyperplane in $\R^k$ passing through the points $N(a_1), \ldots, N(a_k)$, and consider one of the two unit normals $\vec{n} \in \R^k$ to $\Pi_k$.
Let $T_k : \R^k \to \R^k$ be the symmetry with respect to $\Pi_k$, so that for every $y \in \R^k$,
\begin{equation} \label{eq:Tk}
 T_k y = y - 2 (y- N(a_1)) \cdot \vec{n} . 
\end{equation}
Let $\vec{m} = \bar{N} (\vec{n})$, and let $\Pi$ be the affine hyperplane in $\Rn$ with normal $\vec{m}$ passing through $a_1$.
Consider $T : \Rn \to \Rn$ as the symmetry across $\Pi$.
Then, for all $x \in \Rn$,
\begin{equation} \label{eq:T}
 T x = x - 2 (x-a_1) \cdot \vec{m} . 
\end{equation}
Let $a_{k+1}, \ldots, a_n \in \Pi$ be such that the points $a_1, \ldots, a_n$ do not lie in an affine manifold of dimension $n-2$.
Define $A_1^\pm=\{ x \in A_1: \pm \det(x-a_1, a_1-a_2, \ldots, a_1-a_n)>0 \}$.
Then $T(A_1^\pm) = A_1^\mp$, and $A_1^+ \cup A_1^-$ cover $A_1$ up to a set of zero $(n-1)$-measure; see Figure \ref{fi:A1A2}.
\begin{figure}
\begin{center}
\begin{tikzpicture}[line cap=round,line join=round,>=triangle 45,x=0.24999409082687604cm,y=0.2499961199996915cm]
\clip(-13.361615466193403,-7.264385389829523) rectangle (17.439112561143627,7.945850673052942);
\draw [shift={(0.,0.)},line width=1.2pt]  plot[domain=0.4636476090008061:5.81953769817878,variable=\t]({1.*4.47213595499958*cos(\t r)+0.*4.47213595499958*sin(\t r)},{0.*4.47213595499958*cos(\t r)+1.*4.47213595499958*sin(\t r)});
\draw [shift={(6.,0.)},line width=0.4pt]  plot[domain=-2.356194490192345:2.356194490192345,variable=\t]({1.*2.8284271247461903*cos(\t r)+0.*2.8284271247461903*sin(\t r)},{0.*2.8284271247461903*cos(\t r)+1.*2.8284271247461903*sin(\t r)});
\draw [shift={(6.,0.)},line width=0.8pt,dotted]  plot[domain=2.356194490192345:3.9269908169872414,variable=\t]({1.*2.8284271247461903*cos(\t r)+0.*2.8284271247461903*sin(\t r)},{0.*2.8284271247461903*cos(\t r)+1.*2.8284271247461903*sin(\t r)});
\draw [shift={(0.,0.)},line width=0.8pt,dotted]  plot[domain=-0.46364760900080615:0.4636476090008061,variable=\t]({1.*4.47213595499958*cos(\t r)+0.*4.47213595499958*sin(\t r)},{0.*4.47213595499958*cos(\t r)+1.*4.47213595499958*sin(\t r)});
\draw [line width=2.pt] (-7.511866666666666,0.)-- (11.984114285714284,0.);
\draw (-7.331843312693561,2.1605287419922905) node[anchor=north west] {$\Pi$};
\draw (-4.670052001689126,6.370504795111545) node[anchor=north west] {$A_1^+$};
\draw (-4.995985631608037,-3.5976320532417856) node[anchor=north west] {$A_1^-$};
\begin{scriptsize}
\draw [fill=black] (0.,0.) circle (1.5pt);
\draw[color=black] (0.015243928395209855,0.6123439998774682) node {$a_1$};
\draw[color=black] (-2.972481012528135,1.780272840420229) node {$A_1$};
\draw [fill=black] (6.,0.) circle (1.5pt);
\draw[color=black] (5.909210402762172,0.66666627153062) node {$a_2$};
\draw[color=black] (7.267267194090965,1.4000169388481674) node {$A_2$};
\end{scriptsize}
\end{tikzpicture}
\end{center}
\caption{Sets $A_1$, $A_2$, $A_1^+$, $A_1^-$ and $\Pi$\label{fi:A1A2}}
\end{figure}
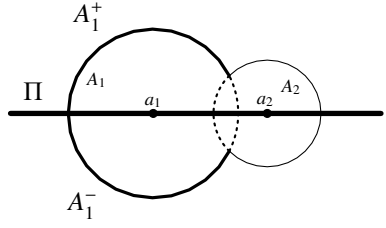
Using the change of variables formula \eqref{eq:changevar sym}, we obtain
\begin{equation}\label{eq:A1-+b}
\begin{split}
 & \int_{A_1^-} \frac{(\bar{M}(\cof (N(x-a_1), N(a_1-a_2), \ldots, N(a_1-a_k))))_1}{ \left| x-a_2 \right|^{n+s+1} \cdots \left| x-a_k \right|^{n+s+1}} \cdot \nu_1 (x) \, d S (x) \\
 & = \int_{A_1^+} \frac{(\bar{M}(\cof (N(Tx-a_1), N(a_1-a_2), \ldots, N(a_1-a_k))))_1}{ \left| Tx-a_2 \right|^{n+s+1} \cdots \left| Tx-a_k \right|^{n+s+1}} \cdot \nu_1(Tx) \, d S (x) .
 \end{split}
\end{equation}
Now, thanks to \eqref{Cof Det}, for $x \in A_1^+$,
\begin{equation}\label{eq:A1-+}
 (\bar{M}(\cof (N(Tx-a_1), N(a_1-a_2), \ldots, N(a_1-a_k))))_1 \cdot \nu_1(Tx)
 = \frac{1}{\epsilon_1} \det (N(Tx-a_1), N(a_1-a_2), \ldots, N(a_1-a_k)) .
\end{equation}
Let $\vec T_k : \R^k \to \R^k$ be the linear map corresponding to the affine map $T_k$, and, analogously, $\vec{T}: \Rn \to \Rn$ the linear map corresponding to $T$.
We notice that $\det \vec T_k = -1$.
Having in mind \eqref{eq:Tk} and \eqref{eq:T}, we find that
\[
 \vec{T}_k y = y - 2 y \cdot \vec{n} , \qquad y \in \R^k 
\]
and
\[
 \vec{T} x = x - 2 x \cdot \vec{m} , \qquad x \in \Rn, 
\]
from which we deduce that $\vec T_k \circ N = N \circ \vec T$.
Thus,
\begin{equation}\label{eq:A1-+c}
\begin{split}
 \det \left( N(Tx - a_1), N(a_1-a_2), \ldots,  N(a_1-a_k) \right) &= \det (N(Tx - Ta_1), N(Ta_1 - Ta_2), \ldots, N(Ta_1 - Ta_k))\\
& = \det (N(\vec{T}(x-a_1)), N(\vec{T}(a_1-a_2)), \ldots, N(\vec{T}(a_1-a_k))) \\
&=\det (\vec{T}_k(N(x-a_1)), \vec{T}_k(N(a_1-a_2)), \ldots , \vec{T}_k (N(a_1-a_k))) \\
&= \det \vec{T}_k \det (N(x-a_1), N(a_1-a_2), \ldots , N(a_1-a_k)) \\
& = -\det (N(x-a_1), N(a_1-a_2), \ldots , N(a_1-a_k)) .
 \end{split}
\end{equation}
Putting together 
 \eqref{eq:A1-+b}, \eqref{eq:A1-+} and \eqref{eq:A1-+c}, we obtain that 
\begin{equation*}
\int_{A_1^-} \frac{\det (N(x-a_1), N(a_1-a_2), \ldots , N(a_1-a_k))}{ \left| x-a_2 \right|^{n+s+1} \cdots \left| x-a_k \right|^{n+s+1}} \, d S (x) 
 =-\int_{A_1^+}\frac{\det (N(x-a_1), N(a_1-a_2), \ldots , N(a_1-a_k))}{ \left| Tx-a_2 \right|^{n+s+1} \cdots \left| Tx-a_k \right|^{n+s+1}} \, d S (x).
\end{equation*}
Consequently, when we define $f : \Rn \setminus \{a_2, \ldots, a_k\} \to \R$ as
\[
 f(y) := \frac{1}{\left( \left| y - a_2 \right| \cdots \left| y -a_k \right| \right) ^{n+s+1}} ,
\]
we have that
\begin{equation}\label{eq:intA1A1+}
\begin{split}
&\int_{A_1} \frac{\det (N(x-a_1), N(a_1-a_2), \ldots , N(a_1-a_k))}{ \left| x-a_2 \right|^{n+s+1} \cdots \left| x-a_k \right|^{n+s+1}}\, d S (x)= \\
&\int_{A_1^+} \det (N(x-a_1), N(a_1-a_2), \ldots , N(a_1-a_k)) \left[ f( x) - f(Tx) \right] d S (x).
\end{split}
\end{equation}

For every $x \in A_1^+$, we join $x$ with $Tx$ by a curve $\g_x$ inside $A_1$, and note that the length of $\g_x$ can be taken to be bounded by $2 \pi \e_1$.
Accordingly, let $\gamma_x:[0,1] \rightarrow A_1$ be of class $C^1$ such that $\gamma_x(0)=x$, $\gamma_x(1)=Tx$ and $|\gamma'_x|$ is constant with $|\gamma'_x|\leq 2 \pi \e_1$.
Then 
\begin{equation}\label{eq:meanvalue}
 |f( x) - f(T x)|=|f(\gamma_x(0))-f(\gamma_x(1))|\leq \int_{0}^{1} |\gamma_x'| \left|\nabla f(\gamma_x(t)) \right|dt  \leq 2 \pi \epsilon_1 \int_{0}^{1}\left|\nabla f(\gamma_x(t)) \right| dt .
\end{equation}
We calculate
\[
 \left| \nabla f(y) \right| = (n+s+1) \left( \left| y - a_2 \right| \cdots \left| y - a_k \right| \right)^{-n-s-2} \sum_{i=2}^k \prod_{\substack{j=2 \\ j\neq i}}^k \left| y - a_j \right| , \qquad y \in \Rn \setminus \{a_2, \ldots, a_k\}.
\]
Now, as $|y-a_j|>\epsilon_j$ for every $y \in A_1$ and $j \in \{ 2,\ldots,k\}$, 
\[
 \left| \nabla f(y) \right| \leq  \frac{n+s+1}{ \left(\epsilon_2 \cdots \epsilon_k \right)^{n+s+2}} \sum_{i=2}^k \prod_{\substack{j=2 \\ j\neq i}}^k \left| y - a_j \right| 
 \leq \frac{n+s+1}{ \left(\epsilon_2 \cdots \epsilon_k \right)^{n+s+2}} \sum_{i=2}^k \prod_{\substack{j=2 \\ j\neq i}}^k (\epsilon_1+\left| a_1 - a_j \right|),
\]
so with \eqref{eq:meanvalue} we obtain that
\begin{equation}\label{eq:fLip}
 |f( x) - f(T x)| \leq 2 \pi \epsilon_1 \frac{n+s+1}{ \left(\epsilon_2 \cdots \epsilon_k \right)^{n+s+2}} \sum_{i=2}^k \prod_{\substack{j=2 \\ j\neq i}}^k (\epsilon_1+\left| a_1 - a_j \right|).
\end{equation}
On the other hand, for all $x \in A_1$,
\begin{equation}\label{eq:detineq}
 \left| \det (N(x-a_1), N(a_1-a_2), \ldots , N(a_1-a_k)) \right| \leq k! \left| x-a_1 \right| \prod_{j=2}^{k} \left| a_1 - a_j \right| = k! \, \epsilon_1 \prod_{j=2}^{k} \left| a_1 - a_j \right| .
\end{equation}
Putting together \eqref{eq:detDH}, \eqref{eq:detDHintA1}, \eqref{eq:h1cofDhnu1}, \eqref{eq:intA1A1+}, \eqref{eq:fLip} and \eqref{eq:detineq}, as well as the fact that the $(n-1)$-dimensional area of $A_1^+$ is bounded by a constant times $\epsilon_1^{n-1}$, we obtain that, for a constant $C>0$ depending on $n$ and $s$,
\[
\left| \int_{\left(\bigcup_{j=1}^{k} B(a_j, \epsilon_j)\right)^c} \frac{ \det (N(x-a_1), N(a_1-a_2), \ldots , N(a_1-a_k))}{|x-a_{1}|^{n+s+1}\cdots|x-a_{k}|^{n+s+1}} dx \right| 
 \leq \frac{C \, \epsilon_1^{1-s}}{\left(\epsilon_2 \cdots \epsilon_k \right)^{n+s+2}} \left( \prod_{j=2}^{k} \left| a_1 - a_j \right| \right) \sum_{i=2}^k \prod_{\substack{j=2 \\ j\neq i}}^k (\epsilon_1+\left| a_1 - a_j \right|) .
\]
The existence of the function $G$ of the statement follows.
\end{proof}

We are in a position to prove the fractional Piola Identity.
Henceforth, $\supp$ denotes the support of a function.

\begin{teo} \label{Piola Identity minors}
Let $k\in \N$ be with $1 \leq k \leq n$.
Consider indices $1 \leq i_1 < \cdots < i_k \leq n$ and $1 \leq j_1 < \cdots < j_k \leq n$ and the functions
\[
 M=M_{i_1 , \ldots , i_k; j_1 , \ldots , j_k} , \quad \bar{M} = \bar{M}_{i_1 , \ldots , i_k; j_1 , \ldots , j_k}
\]
of Definition \ref{de:submatrix}.
Let $u \in C^{\infty}_{c}(\mathbb{R}^n, \Rn)$ and  $s\in (0,1)$.
Then
  \begin{displaymath}
  \Diver^s(\bar{M}(\cof M(D^su)))=0.
  \end{displaymath}
  \end{teo}

  \begin{proof}
Let
\[
 N = N_{j_1 , \ldots , j_k} , \quad \bar{N} = \bar{N}_{j_1 , \ldots , j_k}
\]
be the maps of Definition \ref{de:submatrix}.  
Naturally, $\Diver^s(\bar{M}(\cof M(D^su)))=0$ if and only if
\[
 \diver^s \bar{N} ((\cof M(D^su))_{i_{\ell}}) =0 , \qquad \ell = 1, \ldots, k .
\]
We shall show $\diver^s \bar{N}((\cof M(D^su))_{i_1}) =0$.
The rest of the rows would proceed analogously.
  
Using \eqref{eq:alternativediv}, we have that, for a.e.\ $x \in \Rn$,
   \begin{equation}\label{eq:divcof}
     \frac{(-1)^{k-1}}{c_{n,s}^k} \diver^s \bar{N}((\cof M(D^su))_{i_1})(x)=
    \frac{(-1)^{k-1}}{c_{n,s}^{k-1}}\pv_{x}\int \frac{\bar{N} ((\cof M(D^su))_{i_1}) (x')}{|x'-x|^{n+s+1}}\cdot (x'-x) \, dx'.
    \end{equation}
Now, by \eqref{Cof Det} and \eqref{Alternative gradient def}, we have that for a.e.\ $x, x' \in \mathbb{R}^n$,
   \begin{equation}\label{eq:cofdet}
   \begin{split}
   & \frac{(-1)^{k-1}}{c_{n,s}^{k-1}} \frac{\bar{N} ((\cof M(D^su))_{i_1}) (x')}{|x'-x|^{n+s+1}}\cdot (x'-x) = \frac{(-1)^{k-1}}{c_{n,s}^{k-1}} \frac{(\cof M(D^su))_{i_1} (x')}{|x'-x|^{n+s+1}}\cdot N(x'-x) \\
   & =\frac{(-1)^{k-1}}{c_{n,s}^{k-1}}\frac{\det \left( N(x'-x), N(D^s u_{i_2} (x')) , \ldots ,N(D^s u_{i_k}(x')) \right)}{|x'-x|^{n+s+1}} \\ 
   & = \det \left( \frac{N(x'-x)}{|x'-x|^{n+s+1}}, \pv_{x'} \int \frac{u_{i_2}(y_2) N(x'-y_2)}{|x'-y_2|^{n+s+1}} \, dy_2, \ldots , \pv_{x'} \int \frac{u_{i_k}(y_k) N(x'-y_k)}{|x'-y_k|^{n+s+1}} \, dy_k \right) \\
 & = \lim_{\varepsilon_2 \rightarrow 0} \cdots \lim_{\varepsilon_k \rightarrow 0} f^x_{\e_2, \ldots, \e_k} (x') ,
 \end{split}
   \end{equation}
where for each $x\in \Rn$ and $\e_2, \ldots, \e_k > 0$, we have defined $f^x_{\e_2, \ldots, \e_k} : \Rn \to \R$ by
\[
 f^x_{\e_2, \ldots, \e_k} (x') := \det \left( \frac{N(x'-x)}{|x'-x|^{n+s+1}}, \int_{B(x',\varepsilon_2)^c} \frac{u_{i_2}(y_2) N(x'-y_2)}{|x'-y_2|^{n+s+1}} \, dy_2, \ldots , \int_{B(x',\varepsilon_k)^c} \frac{u_{i_k}(y_k) N(x'-y_k)}{|x'-y_k|^{n+s+1}} \, dy_k \right)
\]
and we have used the continuity of the determinant.
Let $\rho >0$ be such that $\supp u \subset B(x',\rho)$ for all $x' \in \supp u$, and fix $\ell \in \{ 2, \ldots, k \}$.
By odd symmetry, we have that 
   \begin{align*}
      \int_{B(x',\varepsilon_j)^c} u_{i_{\ell}} (y_{\ell}) \frac{N(x'-y_{\ell})}{|x'-y_{\ell}|^{n+s+1}} dy_{\ell} & =
       \int_{B(x',\rho)\setminus B(x',\varepsilon_j)} u_{i_{\ell}}(y_{\ell}) \frac{N(x' - y_{\ell})}{|x'-y_{\ell}|^{n+s+1}}dy_{\ell}
       \\
      &
      = \int_{B(x',\rho)\setminus B(x',\varepsilon_j)}\left( u_{i_{\ell}} (y_{\ell}) - u_{i_{\ell}}(x') \right)\frac{N(x' - y_{\ell})}{|x'-y_{\ell}|^{n+s+1}}dy_{\ell},
      \end{align*}
so, using the fact that $u$ is Lipschitz, we have, for some constant $L>0$, that
   \begin{align*}
   \left| \int_{B(x',\varepsilon_j)^c} u_{i_{\ell}}(y_{\ell})\frac{N(x'- y_{\ell})}{|x'-y_{\ell}|^{n+s+1}}dy_{\ell} \right| & \leq \int_{B(x',\rho)}\frac{\left|u_{i_{\ell}} (y_{\ell}) - u_{i_{\ell}} (x') \right|}{|x'-y_{\ell}|^{n+s}}dy_{\ell} \leq L \int_{B(x',\rho)}\frac{1}{|x'-y_{\ell}|^{n+s-1}}dy_{\ell} \\
   & = L\int_{B(0,\rho)}\frac{1}{|y|^{n+s-1}}dy < \infty .
   \end{align*}
This shows that
\[
 \left| f^x_{\e_2, \ldots, \e_k} (x') \right| \leq \frac{c}{|x'-x|^{n+s}}
\]
for some $c >0$ only depending on $u$ and $n$.
As
\[
 \int_{B(x,\e_1)^c} \frac{1}{|x'-x|^{n+s}} d x' < \infty ,
\]
for any $\e_1>0$, we can apply dominated convergence to conclude that
\[
 \int_{B(x,\e_1)^c} \lim_{\varepsilon_2 \rightarrow 0} \cdots \lim_{\varepsilon_k \rightarrow 0} f^x_{\e_2, \ldots, \e_k} (x') \, d x' = \lim_{\varepsilon_2 \rightarrow 0} \cdots \lim_{\varepsilon_k \rightarrow 0} \int_{B(x,\e_1)^c} f^x_{\e_2, \ldots, \e_k} (x') \, d x' .
\]
Recalling \eqref{eq:divcof} and \eqref{eq:cofdet}, with this we obtain that
\begin{equation}\label{eq:divf}
  \frac{(-1)^{k-1}}{c_{n,s}^k} \diver^s\bar{N}((\cof M(D^su))_{i_1}) (x)=  \lim_{\varepsilon_1 \rightarrow 0} \lim_{\varepsilon_2 \rightarrow 0} \cdots \lim_{\varepsilon_k \rightarrow 0} \int_{B(x,\e_1)^c} f^x_{\e_2, \ldots, \e_k} (x') \, d x' .
\end{equation}
Now for every $\e_1, \ldots, \e_k >0$ we define $D_{\e_1, \ldots, \e_k} := B (x, \e_1) \cup \bigcup_{j=2}^k B (y_j, \e_j)$ and have that, thanks to the multilinearity of the determinant, 
  \begin{align*}
 & \int_{B(x,\e_1)^c} f^x_{\e_2, \ldots, \e_k} (x') \, d x' \\
  &= \int_{B(x,\varepsilon_1)^c} \int_{B(x',\varepsilon_2)^c} \cdots \int_{B(x',\varepsilon_k)^c} \frac{ \det \left( N(x'-x), u_{i_2}(y_2) N(x'-y_2) , \ldots , u_{i_k} (y_k) N(x'-y_k) \right)}{|x'-x|^{n+s+1} |x'-y_2|^{n+s+1} \cdots |x'-y_k|^{n+s+1}} \, dy_2 \cdots dy_k \, dx' \\
  & = \int u_{i_k}(y_k) \cdots \int u_{i_2}(y_2) \int_{D_{\e_1, \ldots, \e_k}^c}
    \frac{ \det \left( N(x'-x), N(x'-y_2) , \ldots , N(x'-y_k) \right)}{|x'-x|^{n+s+1} |x'-y_2|^{n+s+1} \cdots |x'-y_k|^{n+s+1}} \, dx' \, dy_2 \cdots dy_k .
  \end{align*}
Set
\[
 g (x, x', y_2, \ldots, y_k) := \frac{\det \left( N(x'-x), N(x'-y_2) , \ldots , N(x'-y_k) \right)}{|x'-x|^{n+s+1} |x'-y_2|^{n+s+1} \cdots |x'-y_k|^{n+s+1}} .
\]
Then,
\begin{equation}\label{eq:fleqg}
 \left| \int_{B(x,\e_1)^c} f^x_{\e_2, \ldots, \e_k} (x') \, d x' \right| \leq \left\| u \right\|_{\infty}^{k-1} \int_{\supp u} \cdots \int_{\supp u} \left| \int_{D_{\e_1, \ldots, \e_k}^c} g (x, x', y_2, \ldots, y_k)  \, dx' \right| dy_2 \cdots dy_k .
\end{equation}
Thanks to Lemma \ref{Determinante Piola},
\begin{equation}\label{eq:intinterna}
 \left | \int_{D_{\e_1, \ldots, \e_k}^c}
    g (x, x', y_2, \ldots, y_k) \, dx' \right| \leq  \frac{\epsilon_k^{1-s}}{\left(\epsilon_1 \cdots \epsilon_{k-1} \right)^{n+s+2}} G(\epsilon_k, x- y_k, y_2 - y_k, \ldots, y_{k-1} - y_k) ,
\end{equation}
where $G$ is the function that appears therein.
Integrating in \eqref{eq:intinterna}, 
we find that
\[
 \int_{\supp u} \cdots \int_{\supp u} \left | \int_{D_{\e_1, \ldots, \e_k}^c} 
    g (x, x', y_2, \ldots, y_k) \, dx' \right| d y_2 \cdots d y_k \leq h(\e_k, x) \, \frac{\epsilon_k^{1-s}}{\left(\epsilon_1 \cdots \epsilon_{k-1} \right)^{n+s+2}} ,
\]
for some continuous function $h : [0, \infty) \times \Rn \to [0, \infty)$.
Consequently,
\[
 \lim_{\varepsilon_k \rightarrow 0} \int_{\supp u} \cdots \int_{\supp u} \left | \int_{D_{\e_1, \ldots, \e_k}^c}
    g (x, x', y_2, \ldots, y_k) \, dx' \right| d y_2 \cdots d y_k = 0 ,
\]
and, in view of \eqref{eq:divf} and \eqref{eq:fleqg}, we obtain that $\diver^s \bar{N}((\cof M(D^su))_{i_1}) (x) =0$.
\end{proof}


\section[Weak continuity of the determinant]{Weak continuity of the determinant}\label{se:weak}

In this section we prove that any minor (determinant of a submatrix) of $D^s u$ is a weakly continuous mapping in $H^{s,p}$.
We start by expressing a nonlocal integration by parts formula for the minors of $D^s u$ that involves the operator $K_{\varphi}$ of Lemma \ref{Lema operador lineal}. 
Recall that for any $F \in \Rnn$ and $1 \leq i \leq n$ we denote by $F_i$ the $i$-th row of $F$.
 
\begin{lem}\label{le:detintegrationparts}
Let $k \in \mathbb{N}$ be with $1\leq k \leq n$.
Consider indices $1 \leq i_1 < \cdots < i_k \leq n$ and $1 \leq j_1 < \cdots < j_k \leq n$ and the functions
\[
 M=M_{i_1 , \ldots , i_k; j_1 , \ldots , j_k} , \quad \bar{M} = \bar{M}_{i_1 , \ldots , i_k; j_1 , \ldots , j_k} , \quad \tilde{N} = \tilde{N}_{i_1 , \ldots , i_k}
\]
of Definition \ref{de:submatrix}.
Let $p \geq k-1$, $q \geq \frac{p}{p-1}$ and $0 < s < 1$.
Let $u \in H^{s,p}(\mathbb{R}^n, \Rn)$ be such that $\cof M(D^s u) \in L^q (\Rn, \R^{k\times k})$.
Then, $\det M(D^s u) \in L^1_{\loc} (\Rn)$, and for every $\varphi \in C_c^{\infty}(\mathbb{R}^n)$ we have that $\bar{N} (u) \cdot K_{\varphi} (\bar{M} (\cof M (D^s u))) \in L^1 (\Rn)$ and
\begin{equation} \label{Lema continuidad debil det}
\int \det M(D^s u) (x) \, \varphi (x) \, dx = - \frac{1}{k} \int \tilde{N} (u) (x) \cdot K_{\varphi} (\bar{M} (\cof M(D^s u))) (x) \, dx .
\end{equation}
\end{lem}

\begin{proof}
The fact $\det M (D^s u) \in L^1_{\loc} (\Rn)$ is a consequence of formula \eqref{Cof Det} and H\"older's inequality, since $q \geq \frac{p}{p-1}$.
Moreover, $\tilde{N} (u) \cdot K_{\varphi} (\bar{M} (\cof M(D^s u))) \in L^1 (\Rn)$, since $\tilde{N} (u) \in L^p (\Rn, \Rn)$ and $K_{\varphi} (\bar{M} (\cof M(D^s u))) \in L^r (\Rn, \Rn)$ for all $r \in [1, q]$ thanks to Lemma \ref{Lema operador lineal}.

Assume first $u \in C^{\infty}_c (\Rn, \Rn)$ and let $\psi \in C^{\infty}_c (\Rn)$.
Fix $x \in \Rn$ and $i \in \{i_1, \ldots, i_k\}$.
By Lemma \ref{Divergencia producto} and Theorem \ref{Piola Identity minors},
\[
 \div^s \left( \psi \left( \bar{M} (\cof M(D^s u)) \right)_i \right)(x) = K_{\psi} \left( \left( \bar{M} (\cof M(D^s u)) \right)_i^T \right) (x) .
\]
When we apply Theorem \ref{Integracion por partes NL} to the constant function $1$, we obtain from integration of the previous formula that
\[
 0 = \int \div^s \left( \psi \left( \bar{M} (\cof M(D^s u)) \right)_i \right)(x) \, dx = \int K_{\psi} \left( \left( \bar{M} (\cof M(D^s u)) \right)_i^T \right) (x) \, dx .
\]
By Fubini's theorem and the definitions of $K_{\psi}$ and fractional gradient,
\[
 \int K_{\psi} \left( \left( \bar{M} (\cof M(D^s u)) \right)_i^T \right) (x) \, dx = \int D^s \psi (y) \cdot \left( \bar{M} (\cof M(D^s u))\right)_i (y) \, dy .
\]
We thus have the equality
\begin{equation}\label{eq:distributionalPiola}
 \int D^s \psi (y) \cdot \left( \bar{M} (\cof M(D^s u))\right)_i (y) \, dy = 0 .
\end{equation}
Now we assume that $u \in H^{s,p}(\mathbb{R}^n, \Rn)$ with $\cof M (D^s u) \in L^q (\Rn, \R^{k \times k})$, and, again $\psi \in C^{\infty}_c (\Rn)$.
Taking into account Proposition \ref{Theorem properties H^{s,p}}, let $\{ u_j \}_{j \in \N}$ be a sequence in $C^{\infty}_c (\Rn, \Rn)$ converging to $u$ in $H^{s,p}(\mathbb{R}^n, \Rn)$.
Then $M (D^s u_j) \to M (D^s u)$ in $L^p (\Rn, \R^{k \times k})$ and, hence, $\cof M (D^s u_j) \to \cof M (D^s u)$ in $L^{\frac{p}{k-1}} (\Rn, \R^{k \times k})$, so $\bar{M} (\cof M (D^s u_j)) \to \bar{M} (\cof M (D^s u))$ in $L^{\frac{p}{k-1}} (\Rn, \Rnn)$.
Therefore, \eqref{eq:distributionalPiola} holds as well, since $D^s \psi \in L^r (\Rn)$ for all $r \in [1 , \infty]$ (see Lemma \ref{Lema difference quotient bound}).
Now let $\psi \in H^{s,p} (\Rn)$ be of compact support, and let $\{ \psi_j \}_{j \in \N}$ be a sequence in $C^{\infty}_c (\Rn)$ converging to $\psi$ in $H^{s,p}(\mathbb{R}^n)$ such that $\bigcup_{j \in \N} \supp \psi_j$ is bounded.
Then, by Lemma \ref{le:convergence of bounded supported functions}, $D^s \psi_j \to D^s \psi$ in $L^r (\Rn)$ for all $r \in [1, p]$.
As $\bar{M} (\cof M (D^s u)) \in L^q (\Rn, \Rnn)$, we have that \eqref{eq:distributionalPiola} holds as well.
To sum up, formula \eqref{eq:distributionalPiola} is valid for any $u \in H^{s,p}(\mathbb{R}^n)$ with $\cof M (D^s u) \in L^q (\Rn, \R^{k \times k})$ and any $\psi \in H^{s,p}(\mathbb{R}^n)$ of compact support.

We apply \eqref{eq:distributionalPiola} to $\psi = \varphi u_i$, which is in $H^{s,p}(\mathbb{R}^n)$ thanks to Lemma \ref{Gradiente producto}, and has compact support since so does $\f$.
By the formula for $D^s \psi$ given by Lemma \ref{Gradiente producto}, we obtain that
\begin{equation}\label{eq:0=int}
 0 = \int \varphi(y) \, D^s u_i (y) \cdot \left( \bar{M} (\cof M(D^s u))\right)_i (y) \, dy 
 + \int K_{\psi} (u_i I) (y) \cdot \left( \bar{M} (\cof M(D^s u))\right)_i (y) \, dy .
\end{equation}
Using formula \eqref{Cof Det}, the fact $i \in \{i_1, \ldots, i_k\}$ and elementary properties of the functions of Definition \ref{de:submatrix}, we find that for any $F \in \Rnn$,
\[
 F_i \cdot \left( \bar{M} (\cof M (F)) \right)_i = \det M (F) .
\]
Using this and Fubini's theorem, from \eqref{eq:0=int} we arrive at
\[
 0= \int \varphi(y) \det M (D^s u) (y) \, dy + c_{n,s}\int u_i(x) \int \frac{\varphi(x)-\varphi(y)}{|x-y|^{n+s}} \left( \bar{M} (\cof M(D^s u))\right)_i (y) \cdot \frac{x - y}{|x-y|} \, dy \, dx .
\]
We sum this equality for $i= i_1, \ldots, i_k$ and obtain that
\[
 0= k \int \varphi(y) \det M (D^s u) (y) \, dy + c_{n,s}\int \tilde{N} (u) (x) \cdot \int \frac{\varphi(x)-\varphi(y)}{|x-y|^{n+s}} \left( \bar{M} (\cof M(D^s u))\right) (y) \frac{x - y}{|x-y|} \, dy \, dx ,
\]
which is the required formula.
\end{proof} 
  
Now we establish the closedness and continuity properties of the minors of $D^s u$ in the weak topology of $H^{s,p}$.
In the notation of Definition \ref{de:submatrix}\,\emph{\ref{item:M})}, a minor of order $k$ is a function $\mu : \Rnn \to \R$ such that there exist $1 \leq i_1 < \cdots < i_k \leq n$ and $1 \leq j_1 < \cdots < j_k \leq n$ for which $\mu (F) = \det M(F)$ for all $F \in \Rnn$.
Recall the notation $p^*$ of Theorem \ref{th:PoincareSobolev}, and the affine space $H^{s,p}_g$ of \eqref{eq:Hg}.

\begin{teo}\label{th:wcontdet}
Let $p \geq n-1$ and $0 < s < 1$.
Let $g \in H^{s,p} (\Rn)$ and $u \in H^{s,p}_g (\O, \Rn)$.
Let $\{ u_j \}_{j \in \N}$ be a sequence in $H^{s,p}_g (\O, \Rn)$ such that $u_j \weakc u$ in $H^{s,p} (\Rn, \Rn)$.
Then
\begin{enumerate}[a)]
\item\label{item:wcontdetA} If $k \in\N$ with $1 \leq k \leq n-2$ and $\mu$ is a minor of order $k$ then $\mu (D^s u_j) \weakc \mu (D^s u)$ in $L^{\frac{p}{k}} (\Rn)$ as $j \to \infty$.

\item\label{item:wcontcof} If $\cof D^s u_j \weakc \vartheta$ in $L^q (\Rn, \Rnn)$ for some $q \in [1, \infty)$ and $\vartheta \in L^q (\Rn, \Rnn)$ then $\vartheta = \cof D^s u$.

\item\label{item:wcontdet}
Assume $\det D^s u_j \weakc \t$ in $L^{\ell} (\Rn)$ for some $\ell \in [1, \infty)$ and some $\t \in L^{\ell} (\Rn)$.
If $sp < n$ assume, in addition, that $\cof D^s u_j \weakc \cof D^s u$ in $L^q (\Rn, \Rnn)$ for some $q \in ( \frac{p^*}{p^* -1}, \infty)$.
Then $\t = \det D^s u$.
\end{enumerate}
\end{teo}
\begin{proof}
We will prove \emph{\ref{item:wcontdetA})} by induction on $k$.
For $k=1$ the result is trivial.
Assume it holds for some $k \leq n-3$ and let us prove it for $k+1$.
Let $\mu$ be a minor of order $k+1$.
In the notation of Definition \ref{de:submatrix}\,\emph{\ref{item:M})}, $\mu (F) = \det M(F)$ for all $F \in \Rnn$, where $M = M_{i_1, \ldots , i_{k+1} ; j_1, \ldots , j_{k+1}}$ for some $1 \leq i_1 < \cdots < i_{k+1} \leq n$ and $1 \leq j_1 < \cdots < j_{k+1} \leq n$.
Let $\f \in C^{\infty}_c (\Rn)$.
By induction assumption, $\cof M(D^s u_j) \weakc \cof M(D^s u)$ in $L^{\frac{p}{k}} (\Rn, \R^{(k+1) \times (k+1)})$ as $j \to \infty$, so 
$\bar{M}(\cof M(D^s u_j)) \weakc \bar{M}(\cof M(D^s u))$ in $L^{\frac{p}{k}} (\Rn, \Rnn)$.
By Lemma \ref{Lema operador lineal}, $K_{\varphi} (\bar{M}(\cof M(D^s u_j))) \weakc K_{\varphi} (\bar{M}(\cof M(D^su)))$ in $L^r (\Rn, \Rn)$ for every $r \in [1, \frac{p}{k}]$.
By Theorem \ref{Bessel embedding theorem}, $\tilde{N}(u_j) \to \tilde{N}(u)$ in $L^p (\Rn)$, so 
\begin{equation}\label{eq:NK}
 \tilde{N}(u_j) \cdot  K_{\varphi} (\bar{M}(\cof M(D^su_j))) \weakc \tilde{N}(u) \cdot  K_{\varphi} (\bar{M}(\cof M(D^su))) \quad \text{in } L^1 (\Rn)
\end{equation}
since $\frac{k}{p} + \frac{1}{p} \leq 1$.
We apply Lemma \ref{le:detintegrationparts} and, in particular, formula \eqref{Lema continuidad debil det} to conclude that
\begin{equation}\label{eq:detM}
 \int \det M(D^s u_j (x)) \, \varphi (x) \, dx \to \int \det M(D^s u(x)) \, \varphi (x) \, dx .
\end{equation}
This shows that $\det M (D^s u_j) \weakc \det M (D^s u)$ in the sense of distributions.
As $\{ \det M (D^s u_j) \}_{j \in \N}$ is bounded in $L^{\frac{p}{k+1}} (\Rn)$ and $p > k+1$, we have that $\det M (D^s u_j) \weakc \det M (D^s u)$ in $L^{\frac{p}{k+1}} (\Rn)$.

The proof of \emph{\ref{item:wcontcof})} follows the lines of \emph{\ref{item:wcontdetA})}.
Let $\mu$ be a minor of order $n-1$.
In the notation of Definition \ref{de:submatrix}\,\emph{\ref{item:M})}, $\mu (F) = \det M(F)$ for all $F \in \Rnn$, where $M = M_{i_1, \ldots , i_{n-1} ; j_1, \ldots , j_{n-1}}$ for some  $1 \leq i_1 < \cdots < i_{n-1} \leq n$ and $1 \leq j_1 < \cdots < j_{n-1} \leq n$.
Let $\f \in C^{\infty}_c (\O)$.
By part \emph{\ref{item:wcontdetA})}, $\cof M(D^s u_j) \weakc \cof M(D^s u)$ in $L^{\frac{p}{n-2}} (\Rn, \R^{(n-1) \times (n-1)})$, so 
$\bar{M}(\cof M(D^s u_j)) \weakc \bar{M}(\cof M(D^s u))$ in $L^{\frac{p}{n-2}} (\Rn, \Rnn)$.
By Lemma \ref{Lema operador lineal}, $K_{\varphi} (\bar{M}(\cof M(D^s u_j))) \weakc K_{\varphi} (\bar{M}(\cof M(D^su)))$ in $L^r (\Rn, \Rn)$ for every $r \in [1, \frac{p}{n-2}]$.
By Theorem \ref{Bessel embedding theorem}, $\tilde{N}(u_j) \to \tilde{N}(u)$ in $L^p (\Rn)$, so convergence \eqref{eq:NK} is also valid since $\frac{n-2}{p} + \frac{1}{p} \leq 1$.
Thanks to \eqref{Lema continuidad debil det}, we conclude that convergence \eqref{eq:detM} holds.
This shows that $\mu (D^s u_j) \weakc \mu (D^s u)$ in the sense of distributions.
As this is true for every minor $\mu$ of order $n-1$, we obtain that $\cof D^s u_j \weakc \cof D^s u$ in the sense of distributions.
Thanks to the assumption, $\vartheta = \cof D^s u$.

We finally show part \emph{\ref{item:wcontdet})}.
Let $\f \in C^{\infty}_c (\O)$.
Assume first $s p<n$.
By the assumption and Lemma \ref{Lema operador lineal}, $K_{\varphi} (\cof D^s u_j) \weakc K_{\varphi} (\cof D^s u)$ in $L^r (\Rn, \Rn)$ for every $r \in [1, q]$.
By Theorem \ref{Bessel embedding theorem}, $u_j \to u$ in $L^t (\Rn)$ for every $t \in [1, p^*)$, so 
\begin{equation}\label{eq:ucofDu}
 u_j \cdot  K_{\varphi} (\cof D^s u_j) \weakc u_j \cdot K_{\varphi} (\cof D^s u) \quad \text{in } L^1 (\Rn)
\end{equation}
since $\frac{1}{q} + \frac{1}{p^*} < 1$.

Assume now $sp \geq n$.
Then $\{ \cof D^s u_j \}_{j \in \N}$ is bounded in $L^{\frac{p}{n-1}} (\Rn, \Rnn)$ so, thanks to part \emph{\ref{item:wcontcof})}, $\cof D^s u_j \weakc \cof D^s u$ in $L^{\frac{p}{n-1}} (\Rn, \Rnn)$.
By Lemma \ref{Lema operador lineal}, $K_{\varphi} (\cof D^s u_j) \weakc K_{\varphi} (\cof D^s u)$ in $L^r (\Rn, \Rn)$ for every $r \in [1, \frac{p}{n-1}]$.
By Theorem \ref{Bessel embedding theorem}, $u_j \to u$ in $L^t (\Rn)$ for every $t \in [1, \infty)$, so convergence \eqref{eq:ucofDu} holds since $p > n-1$.

In either case, we have convergence \eqref{eq:ucofDu}, so by \eqref{Lema continuidad debil det} we obtain
\[
 \int \det D^s u_j (x) \, \varphi (x) \, dx \to \int \det D^s u(x) \, \varphi (x) \, dx .
\]
This shows that $\det D^s u_j \weakc \det D^s u$ in the sense of distributions, so $\t = \det D^s u$.
\end{proof}

\begin{remark} A natural question is whether the weak continuity of the determinant of the fractional gradient may be concluded as a consequence of the weak continuity of the determinant of the classical gradient. Indeed, one can use the properties of the Riesz potential to give a simpler proof in the case $p>n$. To be precise, in \cite[Prop.\ 2.2]{COMI2019} (see also \cite[Th.\ 1.2]{ShS2015}) it is shown that 
\begin{equation} \label{eq:fractionaltoclassicalgradient}
D^su= D (I_{1-s}\ast u)
\end{equation}
for any $u\in C_c^\infty(\R^n, \R^n)$, where $I_{1-s}(x)=\frac{-c_{n,s}}{n+s-1}|x|^{-(n+1-s)}$. Now, writing the determinant as a divergence \cite[Sect.\ 6]{Ball77} and using \eqref{eq:fractionaltoclassicalgradient}, we have that 
\begin{equation}\label{eq:wcformula}
\int \det D^s u \, \varphi\,dx=-\int ( I_{1-s} \ast u)  \cdot \left(\cof D^su \, D\varphi \right) dx,
\end{equation}
for any $u\in C_c^\infty(\Rn, \R^n)$ and any test function $\varphi\in C_c^\infty(\Rn)$.
By density of $C_c^\infty$ in $H^{s,p}$, equality \eqref{eq:wcformula} holds for any $u\in H^{s,p}(\Rn,\Rn)$. Now, taking into account the Hardy--Littlewood--Sobolev embedding \cite[Th.\ 1,\,b)]{Stein70}, and Theorem \ref{th:PoincareSobolev} it is easy to obtain the weak continuity of $\det D^s u$ in $H^{s,p}(\Rn,\Rn)$ for $p>n$. We do not know whether it is possible to extend the previous argument for $p\ge n-1$ without making use of the fractional Piola identity.
\end{remark}


\section[Existence of minimizers]{Existence of minimizers and equilibrium conditions}\label{se:existence}

In this section we prove the existence of minimizers in $H^{s,p}$ of functionals of the form
\begin{equation}\label{eq:I}
 I (u) := \int W (x, u(x), D^s u (x)) \, dx .
\end{equation}
under natural coercivity and polyconvexity assumptions.
We also derive the associated Euler--Lagrange equation, which is a fractional partial differential system of equations.

We recall the concept of polyconvexity (see, e.g, \cite{Ball77,dacorogna}).
Let $\tau$ be the number of submatrices of an $n \times n$ matrix.
We fix a function $\vec \mu : \Rnn \to \R^{\tau}$ such that $\vec \mu (F)$ is the collection of all minors of an $F \in \Rnn$ in a given order.
A function $W_0 : \Rnn \to \R \cup \{\infty\}$ is polyconvex if there exists a convex $\Phi : \R^{\tau} \to \R \cup \{\infty\}$ such that $W_0 (F) = \Phi (\vec \mu (F))$ for all $F \in \Rnn$.

The existence theorem of this paper is as follows.
Its proof relies on a standard argument in the calculus of variations, once we have the continuity (with respect to the weak convergence) of the minors given by Theorem \ref{th:wcontdet}.

\begin{teo}\label{th:existence}
Let $p \geq n-1$ satisfy $p>1$ and $0 < s < 1$.
Let $W : \Rn \times \Rn \times \Rnn \to \R \cup \{ \infty \}$ satisfy the following conditions:
\begin{enumerate}[a)]
\item $W$ is $\mc{L}^n \times \mc{B}^n \times \mc{B}^{n\times n}$-measurable, where $\mc{L}^n$ denotes the Lebesgue sigma-algebra in $\Rn$, whereas $\mc{B}^n$ and $\mc{B}^{n\times n}$ denote the Borel sigma-algebras in $\Rn$ and $\Rnn$, respectively.

\item $W (x, \cdot, \cdot)$ is lower semicontinuous for a.e.\ $x \in \Rn$.

\item For a.e.\ $x \in \Rn$ and every $y \in \Rn$, the function $W (x, y, \cdot)$ is polyconvex.

\item\label{item:Ecoerc}
There exist a constant $c>0$, an $a \in L^1 (\Rn)$ and a Borel function $h : [0, \infty) \to [0, \infty)$ such that
\[
 \lim_{t \to \infty} \frac{h (t)}{t} = \infty
\]
and
\[
 \begin{cases}
 W (x, y, F) \geq a (x) + c \left| F \right|^p + c \left| \cof F \right|^q + h (\left| \det F \right|)  \quad \text{for some } q > \frac{p^*}{p^* -1},  & \text{if } sp < n ,\\
 W (x, y, F) \geq a (x) + c \left| F \right|^p , & \text{if } sp \geq n ,
 \end{cases}
\]
for a.e.\ $x \in \Rn$, all $y \in \Rn$ and all $F \in \Rnn$.
\end{enumerate}
Let $\O$ be a bounded open subset of $\Rn$.
Let $u_0 \in H^{s,p} (\Rn, \Rn)$.
Define $I$ as in \eqref{eq:I}, and assume that $I$ is not identically infinity in $H^{s,p}_{u_0} (\O, \Rn)$.
Then there exists a minimizer of $I$ in $H^{s,p}_{u_0} (\O, \Rn)$.
\end{teo}
\begin{proof}
Assumption \emph{\ref{item:Ecoerc})} shows that the functional $I$ is bounded below by $\int a$.
As $I$ is not identically infinity in $H^{s,p}_{u_0} (\O, \Rn)$, there exists a minimizing sequence $\{ u_j \}_{j \in \N}$ of $I$ in $H^{s,p}_{u_0} (\O, \Rn)$.
Assumption \emph{\ref{item:Ecoerc})} implies that $\{ D^s u_j \}_{j \in \N}$ is bounded in $L^p (\Rn, \Rnn)$.
Thanks to Theorem \ref{th:PoincareSobolev}, $\{ u_j \}_{j \in \N}$ is bounded in $L^p (\O, \Rnn)$.
As $u_j = u_0$ in $\O^c$ for all $j \in \N$, we also have that $\{ u_j \}_{j \in \N}$ is bounded in $L^p (\Rn, \Rn)$, and, consequently, also in $H^{s,p} (\Rn, \Rn)$.
As $H^{s,p} (\Rn, \Rn)$ is reflexive, we can extract a weakly convergent subsequence.
Using Theorem \ref{Bessel embedding theorem}, we obtain that there exists $u \in H^{s,p}_{u_0} (\Rn, \Rn)$ such that for a subsequence (not relabelled),
\begin{equation}\label{eq:convergence1poly}
 u_j \weakc u \text{ in } H^{s,p} (\Rn, \Rn) \quad \text{and} \quad u_j \to u \text{ in } L^p (\Rn, \Rn) .
\end{equation}
Now, by Theorem \ref{th:wcontdet}, for any minor $\mu$ of order $k \leq n-2$, we have that
\begin{equation}\label{eq:convergence2poly}
 \mu (D^s u_j) \weakc \mu (D^s u) \text{ in } L^{\frac{p}{k}} (\Rn) .
\end{equation}

If $sp <n$ then, by assumption \emph{\ref{item:Ecoerc})}, $\{ \cof D^s u_j \}_{j \in \N}$ is bounded in $L^q (\Rn, \Rnn)$, whereas if $sp \geq n$ we call $q:= \frac{p}{n-1}$ and have that $\{ \cof D^s u_j \}_{j \in \N}$ is bounded in $L^q (\Rn, \Rnn)$.
In either case we have that $q>1$, so for a subsequence $\{ \cof D^s u_j \}_{j \in \N}$ converges weakly in $L^q (\Rn, \Rnn)$ and, by Theorem \ref{th:wcontdet},
\begin{equation}\label{eq:convergence3poly}
 \cof D^s u_j \weakc \cof D^s u \text{ in } L^q (\Rn, \Rnn) .
\end{equation}

If $sp <n$ then, by assumption \emph{\ref{item:Ecoerc})} and de la Vall\'ee Poussin's criterion, $\{ \det D^s u_j \}_{j \in \N}$ is equiintegrable, whereas if $sp \geq n$ we have that $\{ \det D^s u_j \}_{j \in \N}$ is bounded in $L^{\frac{p}{n}} (\Rn)$ and $\frac{p}{n} > 1$.
In either case we have that, for a subsequence $\{ \det D^s u_j \}_{j \in \N}$ converges weakly in $L^{\ell} (\Rn)$ with
\[
 \begin{cases}
 \ell = 1 & \text{if } sp <n , \\
 \ell = \frac{p}{n} & \text{if } sp \geq n ,
 \end{cases}
\]
and, hence, by Theorem \ref{th:wcontdet},
\begin{equation}\label{eq:convergence4poly}
 \det D^s u_j \weakc \det D^s u \text{ in } L^{\ell} (\Rn) .
\end{equation}

Convergences \eqref{eq:convergence1poly}--\eqref{eq:convergence4poly} imply, thanks to a standard lower semicontinuity result for polyconvex functionals (see, e.g., \cite[Th.\ 5.4]{BallCurrieOlver} or \cite[Th.\ 7.5]{FoLe07}), that for any $R>0$,
\[
 \int_{B(0, R)} W(x, u(x), D^s u(x)) \, dx \leq \liminf_{j \to \infty} \int_{B(0, R)} W(x, u_j(x), D^s u_j (x)) \, dx .
\]
Therefore,
\begin{align*}
 \int_{B(0, R)} \left( W(x, u(x), D^s u(x)) - a(x) \right) dx & \leq \liminf_{j \to \infty} \int_{B(0, R)}  \left( W(x, u_j(x), D^s u_j (x)) - a(x) \right) dx \\
 & \leq \liminf_{j \to \infty} \int \left( W(x, u_j(x), D^s u_j (x)) - a(x) \right) dx .
\end{align*}
By monotone convergence,
\[
 \int \left( W(x, u(x), D^s u(x)) - a(x) \right) dx \leq \liminf_{j \to \infty} \int \left( W(x, u_j(x), D^s u_j (x)) - a(x) \right) dx ,
\]
so 
\[
 I(u) \leq \liminf_{j \to \infty} I(u_j) .
\]
Therefore, $u$ is a minimizer of $I$ in $H^{s,p}_{u_0} (\O, \Rn)$ and the proof is concluded.
\end{proof}

Comparing Lemmas \ref{le:fracture} and \ref{le:cavitation} with Theorem \ref{th:existence}, we see that functions exhibiting singularities as those shown in those lemmas are compatible with the existence result of Theorem \ref{th:existence}, in opposition to the case of classical elasticity (see, e.g., \cite{Ball77,Ball82,Ball01,Ball02,HeMo10,BaHeMo17}).
Indeed, for a $u\in H^{s,p}(\Rn)$ of compact support and $p>n$, by H\"older's inequality and Lemma \ref{le:convergence of bounded supported functions}, $\cof D^s u\in L^q (\Rn, \Rnn)$ for every $q \in [1, \frac{p}{n-1}]$ and $\det D^s u \in L^r (\Rn)$ for every $r \in [1, \frac{p}{n}]$.
Take now an $s \in (0,1)$ such that $sp < n$, so that this regime is compatible with cavitation (see Lemma \ref{le:cavitation}).
Considering the function $h$ of Theorem \ref{th:existence} as $h(t) := t^{\frac{p}{n}}$, we see that this map $u$ is compatible with the assumptions of Theorem \ref{th:existence} if and only if $\frac{p}{n-1}>\frac{p^*}{p^*-1}$, so $n^2-np<sp$.
To sum up, in the regime
\[
 p > n , \qquad 0 < s < \frac{n}{p}
\]
a typical cavitation map is compatible with the hypothesis of Theorem \ref{th:existence}.
Similarly, if $p > n$ and $n^2 - np < sp < 1$, i.e., in the regime
\[
 p > n , \qquad 0 < s < \frac{1}{p}
\]
the hypothesis of Theorem \ref{th:existence} are compatible with discontinuities along hypersurfaces.

To finish the article, we explore the equilibrium conditions that minimizers of functional \eqref{eq:I} satisfy. This Euler--Lagrange, or equilibrium, conditions constitute a nonlinear system of fractional PDE, and therefore we are providing an existence result for such kind of systems based on polyconvexity.
To be precise, given $g \in H^{s,p} (\Rn, \Rn)$, the boundary value problem reads as
\begin{equation} \label{Euler-Lagrange equations 1}
\begin{cases}
	\diver^s\left( \frac{\partial W}{\partial F}(x, u , D^s u) \right)=\frac{\partial W}{\partial u}(x, u, D^s u) , & \text{in } \O \\
	u = g & \text{in } \O^c .
\end{cases}
\end{equation}
Inspired by Theorem \ref{Integracion por partes NL}, we define a weak solution of \eqref{Euler-Lagrange equations 1} as a $u \in H^{s,p}_g (\O, \R^n)$ satisfying
\begin{equation}\label{eq: Weak Euler-Lagrange}
 \int \left[ \frac{\partial W}{\partial F}(x, u , D^s u) \cdot D^s v + \frac{\partial W}{\partial u}(x, u, D^s u) \cdot v \right] d x = 0
\end{equation}
for all $v\in C^\infty_c(\R^n, \R^n)$ with $v = 0$ in $\O^c$.

The derivation of \eqref{eq: Weak Euler-Lagrange} for a minimizer $u$ is standard.
For this, we make the assumptions \ref{item:AEL}--\ref{item:BEL}) below, which are slightly adapted from \cite[Conditions 3.22 and 3.33]{dacorogna}, although other sets of assumptions are also possible (see, e.g., \cite[Sect.\ 7]{Ball77} or \cite[Sect.\ 3.4.2]{dacorogna}).

\begin{teo}
Let $W : \Rn \times \Rn \times \R^{n\times n}\to \R$ be a function satisfying
\begin{enumerate}[a)]
\item\label{item:AEL} $W (\cdot, u, F)$ is measurable for every $(u, F) \in \R^n \times \Rnn$ and $W (x, \cdot, \cdot)$ is of class $C^1$ for a.e.\ $x \in \Rn$.

\item\label{item:BEL} There exist an $a \in L^1 (\Rn)$, an $\a \in \R$ with
\[
 \begin{cases}
 \a \in [p, p^*] & \text{if } s p < n , \\
 \a \in [p, \infty ) & \text{if } s p \geq n ,
 \end{cases}
\]
and a $c>0$ such that
\[
 \left| W (x, u, F) \right| + \left| \frac{\p W}{\p u} (x, u, F) \right| + \left| \frac{\p W}{\p F} (x, u, F) \right| \leq a (x) + c \left( |u|^{\a} + |F|^p \right) ,
\]
for a.e.\ $x \in \Rn$ and all $(u, F) \in \Rn \times \Rnn$.
\end{enumerate}
Let $g \in H^{s,p} (\Rn, \Rn)$.
Define $I$ as in \eqref{eq:I}, and let $u$ be a minimizer of $I$ in $H_g^{s,p}(\Omega, \Rn)$.
Then $u$ is a weak solution of \eqref{Euler-Lagrange equations 1}.
\end{teo}
\begin{proof}
Let us fix $v\in C^\infty_c(\R^n, \R^n)$ with $v = 0$ in $\O^c$.
As $u + \tau v \in H^{s,p}_g (\O, \Rn)$ for any $\tau \in \R$, it suffices to show that the derivative of $I (u + \tau \, v)$ exists at $\tau =0$ and equals the left hand side of \eqref{eq: Weak Euler-Lagrange}.
Thanks to the dominated convergence theorem, it suffices to show (see, e.g., \cite[Ch.\ 13, \S 2, Lemma 2.2]{Lang83}) that there exists $G \in L^1 (\Rn)$ such that for every $\tau \in \R$ with $|\tau| < 1$ we have
\begin{equation}\label{eq:Iutv}
 I (u + \tau \, v) < \infty
\end{equation}
and
\begin{equation}\label{eq:Iutv2}
 \left| \frac{d}{d\tau} W(x, u(x) + \tau v(x), D^s u(x) + \tau D^s v(x)) \right|\leq G(x) , \qquad \text{a.e.\ } x \in \Rn .
\end{equation}
Let us check condition \eqref{eq:Iutv}.
Thanks to \ref{item:BEL}) , 
\[
 I (u + \tau \, v) \leq \int a + C \int \left( |u|^{\a} + |D^s u|^p + |v|^{\a} + |D^s v|^p \right)
\]
for some constant $C>0$.
Clearly, the integral of $|v|^{\a}$ is finite since $v\in C^\infty_c(\R^n, \R^n)$, and so is the integral of $|D^s v|^p$ due to Lemma \ref{Lema difference quotient bound}.
In addition, the integral of $|D^s u|^p$ is finite because $u \in H^{s,p} (\Rn, \Rn)$.
Now, by Theorem \ref{th:PoincareSobolev} and the interpolation (or H\"older) inequality, $u \in L^r (\Rn, \Rn)$ for all $r \in [p, p^*]$ if $s p < n$, and for all $r \in [p, \infty)$ if $s p \geq n$.
Therefore, $|u|^{\a} \in L^1 (\Rn)$.
Condition \eqref{eq:Iutv} is thus satisfied.

We now show condition \eqref{eq:Iutv2}.
We have, for $|\tau| < 1$ and a.e.\ $x \in \O$,
\begin{align*}
 & \left| \frac{d}{d\tau} W(x, u(x) + \tau v(x), D^s u(x) + \tau D^s v(x)) \right| \\
 & \qquad \leq \left| \frac{\p W}{\p u} (x, u(x) + \tau v(x), D^s u(x) + \tau D^s v(x)) \right| \left\| v \right\|_{L^{\infty} (\Rn)} + \left| \frac{\p W}{\p F} (x, u(x) + \tau v(x), D^s u(x) + \tau D^s v(x)) \right| \left\| D^s v \right\|_{L^{\infty} (\Rn)} ,
\end{align*}
where we have used Lemma \ref{Lema difference quotient bound} to show that $D^s v \in L^{\infty} (\Rn)$.
Now, by \ref{item:BEL}),
\begin{equation}\label{eq:pWuF}
\begin{split}
 & \left| \frac{\p W}{\p u} (x, u(x) + \tau v(x), D^s u(x) + \tau D^s v(x)) \right| + \left| \frac{\p W}{\p F} (x, u(x) + \tau v(x), D^s u(x) + \tau D^s v(x)) \right| \\
 & \qquad \leq a (x) + C \left( |u (x)|^{\a} + |v (x)|^{\a} + |D^s u (x)|^p + |D^s v (x)|^p \right) ,
\end{split}
\end{equation}
for some constant $C>0$.
As before, the right hand side of \eqref{eq:pWuF} is in $L^1 (\Rn)$, so condition \eqref{eq:Iutv2} is proved.
\end{proof} 

\section*{Acknowledgements} 

This work has been supported by the Spanish {\it Ministerio de Ciencia, Innovaci\'on y Universidades} through projects MTM2017-83740-P (J.C.B. and J.C.), and MTM2017-85934-C3-2-P (C.M.-C.).

\section*{References}

\end{document}